\documentclass[english,10pt]{amsart}
\usepackage[english]{babel}

\usepackage[matrix,arrow]{xy}
\xyoption{all}
\usepackage{amscd,amssymb,amsfonts,amsmath}

\usepackage{graphics}
\usepackage{epsfig}
\usepackage{mathrsfs}

\newcommand\mypagesizel{
\textwidth= 6.5in
\textheight=9in
\voffset-.55in
\hoffset -0.75in
\marginparwidth=56pt
}

\newcommand{\p}[0]{{\mathbb P}}

\newcommand{\Z}{\textup{Z}}

\renewcommand{\phi}{\varphi}

\newcommand{\D}{\Delta}

\newcommand{\sA}{\mathscr{A}}
\newcommand{\sB}{\mathscr{B}}
\newcommand{\sC}{\mathscr{C}}

\newcommand{\sE}{\mathscr{E}}
\newcommand{\sF}{\mathscr{F}}
\newcommand{\sG}{\mathscr{G}}
\newcommand{\sH}{\mathscr{H}}

\newcommand{\sJ}{\mathscr{J}}

\newcommand{\sM}{\mathscr{M}}

\newcommand{\sO}{\mathscr{O}}

\newcommand{\sR}{\mathscr{R}}

\mypagesizel

\newtheorem{thm}{Theorem}[section]

\newtheorem{lemma}[thm]{Lemma}

\newtheorem{prop}[thm]{Proposition}

\newtheorem*{thm*}{Theorem}

\theoremstyle{definition}
\newtheorem{defn}[thm]{Definition}

\newtheorem{const}[thm]{Construction}

\newtheorem{defn-thm}[thm]{Definition-Theorem} 
\newtheorem{defn-lemma}[thm]{Definition-Lemma}

\theoremstyle{remark}
\newtheorem{rem}[thm]{Remark}

\newtheorem*{not-and-def}{Notation and definitions}
\newtheorem{exmp}[thm]{Example}

\numberwithin{equation}{section}

\begin{document}

\title[]{Singular rationally connected threefolds with non-zero pluri-forms}

\author{Wenhao OU} 

\address{Wenhao OU: Institut Fourier, UMR 5582 du
  CNRS, Universit\'e Grenoble 1, BP 74, 38402 Saint Martin
  d'H\`eres, France} 

\email{}

\subjclass[2010]{}

\begin{abstract}
This paper is concerned with singular projective rationally connected threefolds $X$ which carry non-zero pluri-forms, \textit{i.e.} $H^0(X,(\Omega_X^1)^{[\otimes m]}) \neq \{0\}$ for some $m > 0$, where $(\Omega_X^1)^{[\otimes m]}$
is the reflexive hull of $(\Omega_X^1)^{\otimes m}$. If $X$ has $\mathbb{Q}$-factorial terminal singularities, then we show that there is a fibration $p$ from $X$ to $\mathbb{P}^1$. Moreover, there is a natural isomorphism $H^0(X, (\Omega_X^1)^{[\otimes m]}) \cong H^0(\p^1, \sO_{\p^1}(-2m+\sum_{z\in \p^1}[\frac{(m(p,z)-1)m}{m(p,z)}]))$ for all $m>0$, where $m(p,z)$ is the smallest positive coefficient in the divisor $p^*z$.
\end{abstract}

\maketitle

\tableofcontents

\section{Introduction}
\label{Introduction}

Recall that a complex projective variety $X$ is said to be rationally connected if for any two general points in $X$, there
exists a rational curve passing through them (see \cite[Def. IV.3.2 and Prop. IV.3.6]{Kol96}). It is known that if $X$ is a smooth rationally connected variety, then $X$ does not carry any non-zero pluri-form, \textit{i.e.} $H^0(X,$  $(\Omega_X^1)^{\otimes m})=\{0\}$ for $m>0$ (see \cite[Cor. IV.3.8]{Kol96}). It is a conjecture that the converse is also true (see \cite[Conj. IV.3.8.1]{Kol96}). For singular varieties, there are also some analogue results. The following theorems can be found in \cite{GKP12} and \cite{GKKP11}.

\begin{thm}[{\cite[Thm. 3.3]{GKP12}}]
\label{can-fac-thma}
If $X$ is a rationally connected variety with factorial canonical singularities, then $H^0(X,(\Omega_X^1)^{[\otimes m]})=\{0\}$ for $m>0$, where $(\Omega_X^1)^{[\otimes m]}$ is the reflexive hull of $(\Omega_X^1)^{\otimes m}$.
\end{thm}

\begin{thm}[{\cite[Thm. 5.1]{GKKP11}}]
\label{klt-thmb} 
If $(X,D)$ is a klt pair such that $X$ is rationally connected, then $H^0(X,\Omega_X^{[m]})=\{0\}$ for $m>0$, where $\Omega_X^{[m]}$ is the reflexive hull of $\Omega_X^{m}$.
\end{thm}

Note that $\Omega_X^{m}$ is a direct summand of $(\Omega_X^1)^{\otimes m}$, hence $H^0(X,(\Omega_X^1)^{[\otimes m]})=\{0\}$ implies $H^0(X,\Omega_X^{[m]})=\{0\}$. However, there are examples of klt rationally connected varieties whose canonical divisor is effective (see \cite[Exmp. 10]{Tot10} or \cite[Exmp. 43]{Kol08}). Moreover, Theorem \ref{can-fac-thma} is not true if the variety is not factorial. There is a counterexample given in \cite[Example 3.7]{GKP12}. In \cite{Ou13}, we classify all rationally connected surfaces $X$ with canonical singularities such that $H^0(X,(\Omega_X^1)^{[\otimes m]}) \neq \{0\}$ for some $m>0$. We obtain the following theorem. 

\begin{thm}[{\cite[Thm. 1.3]{Ou13}}]
\label{can-surf-thmc}
Let $X$ be a rationally connected surface with canonical singularities. Then $H^0(X,(\Omega_X^1)^{[\otimes m]}) \neq \{0\}$ for some $m>0$ if and only if there is a fibration $p:X \to \p^1$ whose general fibers are isomorphic to $\p^1$ such that $\sum_{z\in \p^1}[\frac{m(p,z)-1}{m(p,z)}]\geqslant 2$, where $m(p,z)$ is the smallest positive coefficient in the divisor $p^*z$.
\end{thm}

In this paper, we will study the case of threefolds. We are interested in the structure of rationally connected threefolds which carry non-zero pluri-forms and we try to find out the source of these forms. If the threefold has terminal singularities, we prove the same result as Theorem \ref{can-surf-thmc} (see Theorem \ref{terminal-3-fold} below). In this case, there is a fibration from the threefold to $\p^1$ and the source of non-zero pluri-forms is the ramification of the fibration. 

\begin{thm}
\label{terminal-3-fold}
Let $X$ be a projective rationally connected threefold with $\mathbb{Q}$-factorial terminal singularities. Then $H^0(X, (\Omega_X^1)^{[\otimes m]}) \neq \{0\}$ for some $m>0$ if and only if there is a fibration $p: X \to \mathbb{P}^1$ whose general fibers are smooth rationally connected surfaces such that $\sum_{z\in \p^1}\frac{m(p,z)-1}{m(p,z)}\geqslant 2$, where $m(p,z)$ is the smallest positive coefficient in the divisor $p^*z$. Moreover, if it is the case, then for all $m>0$, we have $H^0(X, (\Omega_X^1)^{[\otimes m]}) \cong H^0(\p^1, \sO_{\p^1}(-2m+\sum_{z\in \p^1}[\frac{(m(p,z)-1)m}{m(p,z)}]))$, which is of dimension $1-2m+\sum_{z\in \p^1}[\frac{(m(p,z)-1)m}{m(p,z)}]$. 
\end{thm}

\begin{const}
\label{const-thm} Thanks to Theorem \ref{terminal-3-fold}, every rationally connected threefold $X$ with $\mathbb{Q}$-factorial terminal singularities such that $H^0(X, (\Omega_X^1)^{[\otimes m]}) \neq \{0\}$ for some $m>0$ can be constructed as follows. There is a  fibration $q:T \to B$ from a normal threefold  $T$ to a smooth curve $B$ which has positive genus such that $m(q,b)=1$ for all $b \in B$. There is a finite group $G$ which acts on $T$ and $B$ such that $q$ is $G$-equivariant. By taking the quotient by $G$, we have $B/G\cong \p^1$ and $T/G \cong X$. For more details, see \S \ref{dimZ=1}.
\end{const}

If the threefold $X$ has canonical singularities, we also obtain some necessary conditions for existence of non-zero pluri-forms (see \S 3.2 and \S 4). Roughly speaking, if $f_X:X \dashrightarrow X'$ is the result of a MMP for $X$, then there is a fibration from $X'$ to a variety $Z$ of positive dimension such that the forms come from the base $Z$. The difficulty in this situation is that the rational map $X \dashrightarrow Z$ is not always regular. And even it is, it may not be equidimensional. If klt singularities are permitted for our threefolds, then there exist other structures. We give an example of rationally connected threefold of general type in Example \ref{klt-type-gal}.

The main objective of this paper is to proof Theorem \ref{terminal-3-fold}. We will do this in several steps. First, we consider a projective rationally connected  threefold $X$ with $\mathbb{Q}$-factorial canonical singularities which carries non-zero pluri-forms. Since it is rationally connected and with canonical singularities,  its canonical divisor $K_X$ is not pseudo-effective. Hence if $f_X:X\to X'$ is the result of a MMP for $X$, then $K_{X'}$ is not pseudo-effective neither. We have a Mori fibration $p: X' \to Z$ where $Z$ is a normal variety of dimension less than $3$. However, since $X$ carries non-zero pluri-forms, we can show that $H^0(X', (\Omega^1_{X'})^{[\otimes m]}) \neq \{0\}$ for some $m>0$. Then we obtain dim$Z>0$ by \cite[Thm. 2.1]{Ou13}. Hence either dim$Z=1$ or dim$Z=2$. We will treat this two cases separately in \S\ref{dimZ=1} and \S\ref{dimZ=2}. If dim$Z=1$, then $Z\cong \p^1$ and we will show that $K_{\p^1} + \sum_{z\in \p^1}\frac{m(p,z)-1}{m(p,z)} z$ is effective, which is the same condition as in Theorem \ref{terminal-3-fold}. If dim$Z=2$, then we will define a $\mathbb{Q}$-divisor $\D$ on $Z$ (see Theorem \ref{dim2-picard-1} for the definition). We will prove that either $K_Z+\D$ is an effective $\mathbb{Q}$-divisor or there is a fibration $X'\to \p^1$ such that we can reduce to the situation of \S\ref{dimZ=1}. In the last section, we assume that the variety $X$ has terminal singularities. In this case, there is always a fibration from $X'$ to $\p^1$ and we are always in the situation of \S\ref{dimZ=1}. This fibration induces a dominant rational map from $X$ to $\p^1$. In the end, we will prove that this rational map is regular and complete the proof of Theorem \ref{terminal-3-fold}

\section{Preliminaries}
\label{Preliminaries}

Throughout this paper, we will work over $\mathbb{C}$, the field of complex numbers. Unless otherwise specified, every
variety is an integral $\mathbb{C}$-scheme of finite type. A curve is a variety of dimension 1, a surface is a
variety of dimension 2 and a threefold is a variety of dimension 3. For a normal variety $X$, let $K_X$ be a canonical divisor of $X$. We denote the sheaf of K\"{a}hler differentials by $\Omega_{X}^1$. Denote
$\bigwedge^p\Omega_{X}^1$ by $\Omega_{X}^p$ for $p\in \mathbb{N}$. Let $\Omega_X^{[p]}$ (resp. $(\Omega_X^{1})^{[\otimes p]}$) be the reflexive hull of $\Omega_X^p$ (resp. $(\Omega_X^1)^{\otimes p}$). We say that a normal variety $X$ carries non-zero pluri-forms if $H^0(X,(\Omega_X^1)^{[\otimes m]}) \neq \{0\}$ for some $m>0$. Let Pic$(X)_{\mathbb{Q}}=$Pic$(X) \otimes \mathbb{Q}$, where Pic$(X)$ is the Picard group of $X$. A $\mathbb{Q}$-divisor $\D$ on $X$ is called effective if there is a positive integer $k$ such that $kD$ is a divisor and $\sO_X(kD)$ has a non-zero global section.

A fibration $p:X\to Z$ between normal quasi-projective varieties is a dominant proper morphism such that every fiber is connected, that is $p_*\sO_X \cong \sO_Z$. If $X \to Z$ is just a proper morphism, then we have the Stein factorisation $X\to V \to Z$ such that $X\to V$ is a fibration and $V\to Z$ is finite (see \cite[Cor. III. 11.5]{Har77}). A Fano fibration is a fibration $X\to Z$ such that $-K_X$ is relatively ample. A Mori fibration is a Fano fibration such that the relative Picard number is $1$.

Let $p:X \to Z$ be a morphism and $D$ be a prime $\mathbb{Q}$-Cartier Weil divisor in $Z$. Let $k$ be a positive integer such that $kD$ is Cartier. Denote the $\mathbb{Q}$ divisor $\frac{1}{k}p^*(kD)$ by $p^*D$. We define the multiplicity $m(p,D)=$min$\{$coefficient of $E$ in $p^*D \ | \ E$ is an irreducible component of $p^*D$ which dominates $D\}$. There are only finitely many divisors $D$ such that $m(p,D)>1$. Moreover, if $p:X \to Z$ is a morphism from a normal variety to a smooth curve, then we define the ramification divisor by $R=\sum_{z \in Z} p^*z - (p^*z)_{red}$ where $ (p^*z)_{red}$ is the sum of the irreducible components of $p^*z$.

\subsection{Reflexive sheaves on normal varieties}
\label{Reflexive sheaves on normal varieties}
In this subsection, we gather some properties of reflexive sheaves. For a coherent sheaf $\mathscr{F}$ on a normal variety $X$, we denote by $\mathscr{F}^{**}$ the double dual of $\mathscr{F}$. The sheaf $\sF$ is reflexive if and only if $\sF \cong \sF^{**}$. In particular, $\sF^{**}$ is reflexive and we also call it the reflexive hull of $\sF$. Denote $(\mathscr{F}^{\otimes m})^{**}$  by $\mathscr{F}^{[\otimes m]}$ and $(\bigwedge^m \mathscr{F})^{**}$  by $\mathscr{F}^{[\wedge m]}$ for any $m>0$. For two coherent sheaves $\mathscr{F}$ and $\mathscr{G}$, let $\mathscr{F} [\otimes] \mathscr{G} = (\mathscr{F} \otimes \mathscr{G})^{**}$. The following proposition is an important criterion for reflexive sheaves on normal varieties.

\begin{prop}[{\cite[Prop. 1.6]{Har80}}]
\label{refle-cod-2}
Let $\mathscr{F}$ be a coherent sheaf on a normal variety $X$. Then $\mathscr{F}$ is reflexive
if and only if $\mathscr{F}$ is torsion-free and for each open $U \subseteq X$ and each closed subset $Y\subseteq U$ of
codimension at least 2, $\mathscr{F}|_U \cong j_*(\mathscr{F}|_{U\setminus Y})$, where $j:U\setminus Y \to U$ is the
inclusion map.
\end{prop}

As a corollary of this proposition, we prove the following lemma.

\begin{lemma}
\label{inj-mmp}
If $\phi:X \dashrightarrow X'$ is a birational map between normal projective varieties such that $\phi^{-1}$ does not contract any divisor, then we have a natural injection $H^0(X,(\Omega_X^1)^{[\otimes m]}) \hookrightarrow H^0(X',(\Omega_{X'}^1)^{[\otimes m]})$ for any integer $m>0$.
\end{lemma}

\begin{proof}
Since the birational map $\phi^{-1}: X' \dashrightarrow X$ does not contract any divisor, it induces an isomorphism from an open subset $W'$ of $X'$ such that codim$X'\backslash W \geqslant 2$  onto an open subset $W$ of $X$. By Proposition \ref{refle-cod-2}, we have $H^0(X',(\Omega_{X'}^1)^{[\otimes m]}) \cong H^0(W',(\Omega_{W'}^1)^{[\otimes m]})$ for any $m>0$. Moreover, since $W$ is an open subset of $X$, we obtain $H^0(X,(\Omega_{X}^1)^{[\otimes m]}) \hookrightarrow H^0(W,(\Omega_{W}^1)^{[\otimes m]}) \cong H^0(W',(\Omega_{W'}^1)^{[\otimes m]}) \cong H^0(X',(\Omega_{X'}^1)^{[\otimes m]})$ for all $m>0$.
\end{proof}

The proof of the following lemma is left to the reader.	

\begin{lemma}
\label{long-filtration}
Let $0\to \mathscr{A} \to \mathscr{B} \to \mathscr{C} \to 0$ be an exact sequence of locally free sheaves on a variety $X$. Then for any $m> 0$, we have a filtration of locally free sheaves $\mathscr{B}^{\otimes m} = \mathscr{R}_0 \supseteq \cdots \supseteq \mathscr{R}_{m+1}=0$ such that $\mathscr{R}_{i}/\mathscr{R}_{i+1}$ is isomorphic to the direct sum of copies of $\mathscr{A}^{\otimes i} \otimes \mathscr{C}^{\otimes m-i}$ for all $0 \leq i \leq m$. 
\end{lemma}

From this lemma, we can deduce the  following lemma which is very important in the paper.

\begin{lemma}
\label{exact-section}
Let $\mathscr{A} \to \mathscr{B} \to \mathscr{C}$ be a complex of coherent sheaves on a normal variety $X$. Assume that there is an open subset $W$ of $X$ with $\mathrm{codim}X\backslash W \geqslant 2$  such that we have an exact sequence of locally free sheaves $0\to \mathscr{A}|_W \to \mathscr{B}|_W \to \mathscr{C}|_W \to 0$ on $W$. If $H^0(X, \sA^{\otimes r} [\otimes] \sC^{\otimes t}) = \{0\}$ for all $t>0$ and $r \geqslant 0$, then   $H^0(X,\sB^{[\otimes m]}) \cong H^0(X,\sA^{[\otimes m]})$ for all $m >0$.
\end{lemma}

\begin{proof}
On $W$, we have $H^0(W, \sA|_W^{\otimes r} \otimes \sC|_W^{\otimes t}) = \{0\}$ for all $t>0$ and $r \geqslant 0$ by Proposition \ref{refle-cod-2}. Fix $m > 0$, then we have a filtration $\mathscr{B}|_W^{\otimes m} = \mathscr{R}_0 \supseteq \cdots \supseteq \mathscr{R}_{m+1}=0$ such that $\mathscr{R}_{i}/\mathscr{R}_{i+1}$ is isomorphic to the direct sum of copies of $\mathscr{A}|_W^{\otimes i} \otimes \mathscr{C}|_W^{\otimes m-i}$ for all $0 \leq i \leq m$. Since $H^0(X, \sA|_W^{\otimes r} \otimes \sC|_W^{\otimes t}) = \{0\}$ for all $t>0$ and $r \geqslant 0$, we have $H^0(W, \sR_i) \cong H^0(W, \sR_{i+1})$ for $0\leqslant i \leqslant m-1$. Thus $H^0(W,\sB|_W^{\otimes m}/\sA|_W^{\otimes m}) =\{0\}$ and $H^0(W,B|_W^{\otimes m}) \cong H^0(W,A|_W^{\otimes m})$. By Proposition \ref{refle-cod-2}, we have $H^0(X,\sB^{[\otimes m]}) \cong H^0(X,\sA^{[\otimes m]})$.
\end{proof}

One of the applications of the lemma above is the following, which gives a relation between pluri-forms and fibrations.

\begin{lemma}
\label{fibration-section}
Let $p:X\to Z$ be an equidimensional morphism between normal varieties. Assume that general fibers of $p$ do not carry any non-zero pluri-form. Then $H^0(X,(\Omega_X^1)^{[\otimes m]}) \cong H^0(X, ((p^*\Omega_Z^1)^{sat})^{[\otimes m]})$ for $m>0$, where $(p^*\Omega_Z^1)^{sat}$ is the saturation of the image of $p^*\Omega_Z^1$ in $\Omega_X^{[1]}$.
\end{lemma}

\begin{proof}
We have an exact sequence of coherent sheaves $0 \to \sF \to \Omega_X^{[1]} \to \sG \to 0$ on $X$, where  $\sF = (p^*\Omega_Z^1)^{sat}$ and $\sG$ is a torsion free sheaf.  Let $V$ be the smooth locus of $Z$  and let $W=p^{-1}(V)$. If $W_{ns}$ is the largest open subset of $W$ on which $\sF$, $\Omega_X^1$ and $\sG$ are locally free, then codim$X\backslash W_{ns} \geqslant 2$ and we have an exact sequence $0\to \sF|_{W_{ns}} \to \Omega_{W_{ns}}^{1} \to \sG|_{W_{ns}} \to 0$ of locally free sheaves on $W_{ns}$.

If $F$ is a general fiber of $p|_{W_{ns}}$, then $\sF|_F$ is the direct sum of $\sO_F$ and $\sG|_F$ is isomorphic to $\Omega_F^1$. Since general fibers of $p$ do not carry any non-zero pluri-form, neither does $F$ by Proposition \ref{refle-cod-2}. Hence we have $H^0(W_{ns},\sF|_{W_{ns}}^{\otimes r} \otimes \sG|_{W_{ns}}^{\otimes t}) = \{0\}$ for all $t>0$ and $r\geqslant 0$. By Lemma \ref{exact-section}, we have $H^0(W_{ns},\sF|_{W_{ns}}^{\otimes m}) \cong H^0(W_{ns},(\Omega_{W_{ns}}^1)^{\otimes m})$ for all $m>0$. Since codim$X\backslash W_{ns} \geqslant 2$, by Proposition \ref{refle-cod-2}, we have $H^0(X,(\Omega_X^1)^{[\otimes m]}) \cong H^0(X, \sF^{[\otimes m]})$ for all $m>0$.
\end{proof}

If $X$ is a normal variety and $E$ is a Weil divisor on $X$, then $E$ is a Cartier divisor on the smooth locus $X_{ns}$ of $X$ and it induces an invertible sheaf $\sO_{X_{ns}}(E)$ on $X_{ns}$. Let $\sO_X(E)$ be the push-forward of  $\sO_{X_{ns}}(E)$ on $X$. Then by Proposition \ref{refle-cod-2}, $\sO_X(E)$ is a reflexive sheaf on $X$ since $X$ is smooth in codimension $1$. Conversely, if $\sF$ is a reflexive sheaf of rank $1$ on $X$, it is an invertible sheaf in codimension $1$. There is a Weil divisor $D$ on $X$ such that $\sF \cong \sO_X(D)$. If $X$ is $\mathbb{Q}$-factorial, then for any $1$-cycle $\alpha$ on $X$, we can define the intersection number $\sF \cdot \alpha = D \cdot \alpha = \frac{1}{k}(kD) \cdot \alpha$, where $k$ is a positive integer such that $kD$ is Cartier. This expression is independent to the choice of $D$ (see \cite[Appendix to \S 1]{Reid79} for more details).

\subsection{Minimal model program}
\label{Minimal model program}

We will recall some basic definition and  properties of the minimal model program (MMP for short).  A pair $(X,\D)$ consists of a normal quasi-projective variety $X$ and a boundary $\D$, \textit{i.e.} a $\mathbb{Q}$-Weil divisor divisor $\D=\Sigma_{j=1}^kd_j D_j$ on $X$ such that the $D_j$'s are pairwise distinct prime divisors and all $d_j$ are contained in $[0,1]$. Recall the definition of singularities of pairs.

\begin{defn}[{\cite[Def. 2.34]{KM98}}]
\label{def-sing-pair}
Let $(X,\D)$ be a pair with $\D=\Sigma_{j=1}^kd_j D_j$. Let $r:\widetilde{X} \to X$ be a log resolution of singularities of $(X, \D)$. Assume that $K_X+\D$ is $\mathbb{Q}$-Cartier, then we can write $K_{\widetilde{X}}+r_*^{-1}\D = r^*(K_X+\D)+\Sigma a_iE_i$, where $E_i$ are $r$-exceptional divisors. We call $a_i$ the discrepancy of $E_i$ with respect to $(X,\D)$. Then $(X,\D)$ is terminal (resp. canonical, klt) if $a_i>0$ (resp. $a_i\geqslant 0$, $a_i> -1$ and $d_j<1$ for all $j$) for all $i$.
\end{defn}

\begin{rem}
\label{rem-sing-codim-2}
If $X$ is terminal, then it is smooth in codimension $2$. If $X$ is canonical, then $K_X$ is Cartier in codimension $2$ (see \cite[Cor. 5.18]{KM98}). 
\end{rem}

For a klt pair $(X,\D)$ such that $X$ is $\mathbb{Q}$-factorial, we can run a $(K_X+\D)$-MMP and obtain a sequence of rational maps $X=X_0 \dashrightarrow X_1 \dashrightarrow \cdots$ such that $(X_i,\D_i)$ is a klt pair and $X_i$ is $\mathbb{Q}$-factorial, where $\D_i$ is the strict transform of $\D$. Every elementary step in a MMP is either a divisorial contraction which is a morphism contracting an irreducible divisor or a flip which is an isomorphism in codimension $1$. More generally, if $h:X \to T$ is a morphism, we can run a $h$-relative MMP such that for all $i$, there is a morphism $h_i:X_i \to T$ with $h_i \circ \phi_i = h$, where $\phi_i$ is the birational map $X \dashrightarrow X_i$. For more details on MMP, we refer to \cite[\S 3]{KM98}. One of the most important problems in MMP is that if the sequence of rational maps above terminates. If it does, we will get the result of this $h$-relative MMP $X \dashrightarrow X'$. Let $\D'$ be the direct image of $\D$. Then either  $K_{X'}+\D'$ is relative nef over $T$ or we have a $(K_{X'}+\D')$-Mori fibration $g:X' \to Z$ over $T$ such that $-(K_{X'}+\D')$ is $g$-ample. Thanks to \cite[Thm. 1]{Kaw92}, we know that any MMP for a klt pair $(X,\D)$ such that dim$X\leqslant 3$ terminates.

\begin{lemma}
\label{iso-can-center}
Let $\phi: X \dashrightarrow X'$ be an extremal divisorial contraction or a flip such that both $X$ and $X'$ are projective. Assume that $X$ has  $\mathbb{Q}$-factorial canonical singularities. If $Y$ is an irreducible closed subvariety in $X'$ such that it is the center of a divisor $E$ over $X'$ which has discrepancy $0$, then $\phi^{-1}$ is a morphism around the generic point of $Y$.
\end{lemma}

\begin{proof}
Assume the opposite. Then the discrepancy of $E$ in $X$ is strictly smaller than the one in $X'$ (see \cite[Lem. 3.38]{KM98}). This implies that $X$ does not have canonical singularities along the center of $E$  in $X$ which is a contradiction.
\end{proof}

\begin{prop}
\label{iso-graph}
Let $X$ be a projective threefold which has at most canonical singularities. Let $X'$ be the result of a MMP for $X$ and denote the birational map $X \dashrightarrow X'$ by $f_X$. Let $\Gamma$ be the normalisation of the graph of $f_X$. Then there is a natural isomorphism $H^0(X,(\Omega_X^1)^{[\otimes m]}) \cong H^0(\Gamma,(\Omega_{\Gamma}^1)^{[\otimes m]})$ for all $m>0$.
\end{prop}

\begin{proof}
Since there is a natural birational projection $pr_1: \Gamma\to X$, we have an injection $H^0(\Gamma,(\Omega_{\Gamma}^1)^{[\otimes m]}) \hookrightarrow H^0(X,(\Omega_X^1)^{[\otimes m]})$ by Lemma \ref{inj-mmp}. Let $pr_2:\Gamma\to X'$ be the natural projection. Now let $\sigma_X \in H^0(X,(\Omega_X^1)^{[\otimes m]})$ be a non-zero element. Since $X$, $\Gamma$ and $X'$ are birational, $\sigma_X$ induces a rational section $\sigma_{\Gamma}$ of $(\Omega_{\Gamma}^1)^{[\otimes m]}$ and an element $\sigma_{X'}$ of $H^0(X',(\Omega_{X'}^1)^{[\otimes m]})$ (see Lemma \ref{inj-mmp}). In order to prove that $\sigma_{\Gamma}$ is a regular section it is enough to prove that $\sigma_{\Gamma}$ does not have pole along any $pr_1$-exceptional divisor. Let $E$ be an exceptional divisor for $pr_1$. Let $C \subseteq E$ be a curve which is exceptional for $pr_1$, then $C$ is not contracted by $pr_2$ since the graph of $f_X$ is included in $X\times X'$ and the normalisation map is finite. Hence $Y=pr_2(E) \subseteq X'$ is a curve and $f_X^{-1}$ is not regular around the generic point of $Y$. By Lemma \ref{iso-can-center}, $X'$ has terminal singularities around the generic point of $Y$. Hence it is smooth around the generic point of $Y$ since codim$_{X'}Y=2$ (see \cite[Cor. 5.18]{KM98}). Hence $\sigma_{\Gamma}$ does not have pole along $E$ ($\sigma_{\Gamma}$ is the pullback of $\sigma_{X'}$ by $pr_2$).
\end{proof}

\subsection{Examples}

In this subsection, we will give some examples of rationally connected varieties which carry non-zero pluri-forms. First we will give an example of such varieties with terminal singularities. For the construction of this example, we use the method of Theorem \ref{thm_curve}.

\begin{exmp}
\label{exmp-curve-1}
Let $C_1=\{[a:b:c] \in \mathbb{P}^2 \ | \  a^3+(a+c)b^2+c^3=0\}$ be a smooth elliptic curve in $\mathbb{P}^2$. Let $X_1=\{([a:b:c],[x:y:z:t]) \in \mathbb{P}^2 \times \mathbb{P}^3 \ | \  a^3+(a+c)b^2+c^3=0, \ (a^2+2b^2+c^2)x^2 + (a^2+3b^2+c^2)y^2 + (a^2+4b^2+3c^2)z^2 + (a^2+5b^2+6c^2)t^2=0 \}$. Then $X_1$ is a smooth threefold and we have an induced fibration $\pi_1:X_1 \to C_1$ such that general fibers of $\pi_1$ are smooth quadric surfaces which are Fano surfaces. Moreover, $X_1$ has Picard number $2$ by the Lefschetz theorem (see \cite[Exmp. 3.1.23]{Lar04}). Hence $\pi_1$ is a Mori fibration. Moreover, since $\pi_1$ is smooth, we have $H^0(X_1, (\Omega^1_{X_1})^{\otimes 2}) \cong H^0(C_1, (\Omega^1_{C_1})^{\otimes 2})$ by Lemma \ref{fibration-section}. Let $G$ be the group $\mathbb{Z}/2\mathbb{Z}$ and let $g\in G$ be the generator. We have an action of $G$ on $\mathbb{P}^2 \times \mathbb{P}^3$ defined by $g\cdot ([a:b:c],[x:y:z:t])=([a:-b:c],[-x:-y:z:t])$. This action induces an action of $G$ on $X_1$ and an action of $G$ on $C_1$ such that $\pi_1$ in $G$-equivariant. We have $C_1/G=\mathbb{P}^1$. The action of $G$ on $X_1$ is \'etale in codimension 2 and it has exactly 16 fix points. Thus $X=X_1/G$ is a threefold which has at most $\mathbb{Q}$-factorial terminal singularities by \cite[(5.1)]{Reid87}. We have a fibration $\pi : X \to \mathbb{P}^1$ induced by $\pi_1$. Moreover, the $\mathbb{Q}$-dimension of Pic$(X)_{\mathbb{Q}}$ is not larger than the Picard number of $X_1$. Hence $\pi$ is a Mori fibration since general fibers of $\pi$ are smooth quadric surfaces. Hence $X$ is rationally connected by \cite[Thm. 1.1]{GHS03}. In addition, $H^0(X, (\Omega^1_{X})^{[\otimes 2]}) \cong H^0(X_1, (\Omega^1_{X_1})^{\otimes 2})^G$ since $X_1 \to X$  is \'etale in codimension $1$. Moreover  $H^0(C_1, (\Omega^1_{C_1})^{\otimes 2})^G \cong H^0(\mathbb{P}^1, \mathscr{O}_{\mathbb{P}^1}) \cong \mathbb{C}$ by \cite[Prop. 6.2]{Ou13}. Hence we have $H^0(X, (\Omega^1_{X})^{[\otimes 2]}) \cong \mathbb{C}$.  
\end{exmp}

\begin{rem}
\label{non-iso-rem}
If $X$ is a normal rationally connected threefold and if $X'$ is the result of a MMP for $X$, then $H^0(X,(\Omega_X^1)^{[\otimes m]})$ may be strictly included in $H^0(X',(\Omega_{X'}^1)^{[\otimes m]})$ for some $m>0$. For example, let $X' \to \mathbb{P}^1$ be the variety described in the example above. Then $X'$ carries non-zero pluri-forms. Let $X \to X'$ be a resolution of singularities of $X'$. Then $X'$ is the result of a MMP for $X$ since $X'$ has terminal singularities. However, since $X$ is rationally connected and smooth, we have $H^0(X, (\Omega^1_{X})^{[\otimes m]}) = \{0\}$ for any $m>0$ (see \cite[Cor. IV.3.8]{Kol96}).
\end{rem}

We will give two examples of rationally connected threefolds with canonical singularities which carry non-zero pluri-forms. In both cases, there is always a Mori fibration from the threefold $X$ to a surface $Z$. In Example \ref{KB-pf}, the base $Z$ is a Calabi-Yau surface with klt singularities which carries non-zero pluri-forms. In Example \ref{over-surf-pf}, the base $Z$ is isomorphic to $\p^2$ which does not carry any non-zero pluri-form.

\begin{exmp}
\label{KB-pf}
Let $C$ be the  curve  in $\mathbb{P}^2$ defined by $C=\{[x_1:x_2:x_3] \in \mathbb{P}^2 \ | \  x_1^3+x_2^3+x_3^3=0   \}$. There exists an action of group $G=\mathbb{Z}/3\mathbb{Z}$ on $C$ defined by $g\cdot [x_1:x_2:x_3] = [\xi x_1:x_2:x_3]$ where $g$ is a generator of $G$ and $\xi$ is a primitive 3rd root of unity. Let $Z_1=C\times C$. Then there is an action of $G$ on $Z_1$ induced by the one of $G$ on $C$. We have $Z=Z_1/G$ is a klt rationally connected surface such that $K_Z$ is a torsion (see \cite[Example 10]{Tot10}). In particular, $K_Z$ is an effective  $\mathbb{Q}$-Cartier divisor. Let $X_1=\mathbb{P}^1 \times Z_1$ and define an action of $G$ on $\mathbb{P}^1$ by $g \cdot [y_1:y_2] \to [\xi y_1:y_2]$ where $[y_1:y_2]$ are homogeneous coordinates of $\mathbb{P}^1$. Then we have an  action of $G$ on the smooth threefold $X_1$ and $X=X_1/G$ has canonical singularities (see \cite[Thm. 3.1]{Reid79}). We also have an Mori fibration $X \to Z$ whose general fibers are isomorphic to $\p^1$. Since $Z$ is rationally connected, so is $X$ by \cite[Thm. 1.1]{GHS03}. Moreover, since $K_Z$ is effective, we have $H^0(X,(\Omega^1_{X})^{[\otimes m]}) \neq \{0\}$ for some $m>0$ (see \S 4.2). 
\end{exmp}

In the example above, the non-zero pluri-forms  come from $K_Z$, the canonical divisor of the base surface. In the following example, $-K_Z$ is ample. However $X$ still carries  non-zero pluri-forms. These forms come from the multiple fibers of the fibration $X \to Z$. Before giving the example, we will first introduce a method to construct  fibers with high multiplicities.

\begin{const}
\label{const-non-reduced-fib}
We want to construct a fibration $p$ from a normal threefold $T$ with canonical singularities to a smooth surface $S$ such that $p^*C=2(p^*C)_{red}$, where $C$ is a smooth curve in $S$.

Let $T_0=\mathbb{P}^1 \times S$ where $S$ is a smooth projective surface. Let $pr_1$, $pr_2$ be the natural projections $T_0 \to \p^1$ and $T_0 \to S$. Let $C \subseteq S$ be a smooth curve and let $C_0=\{z\}\times C$ be a section of $pr_2$ over $C$ in $T_0$, where $z$ is a point of $\p^1$. Let $E_0 = pr_2^* C \subseteq T_0$ be a smooth divisor. We will perform a sequence of birational transformation of $T_0$.

First we blow up $C_0$, we obtain a morphism $T_1 \to T_0$ with  exceptional divisor $E_1$. Denote still by $E_0$ the strict transform of $E_0$ in $T_1$. If $F$ is a fiber of  $T_1 \to S$ over a point of $C$, then $F=F_1+F_0$ where $F_i \subseteq E_i$ are rational curves and $K_{T_1} \cdot F_i = -1$, $E_i \cdot F_i = -1$.

Now we blow up $C_1$, the intersection of $E_0$ and $E_1$. We obtain a morphism $T_2 \to T_1$ with exceptional divisor $E_2$. Denote still by $E_0$, $E_1$ the strict transforms of $E_0$ and $E_1$ in $T_2$. If $F$ is any set-theoretic fiber of $T_2 \to S$ over a point of $C$, then $F=F_0+F_1+F_2$ where $F_i \subseteq E_i$ are rational curves and $K_{T_2} \cdot F_i = 0$, $E_i \cdot F_i = -2$ for $i=0,1$, $K_{T_2} \cdot F_2 = -1$, $E_2 \cdot F_2 = -1$. Let $p_2$ be the fibration $T_2 \to S$.

We will blow down $E_1$ and $E_0$ in $T_2$. Let $H$ be an ample $\mathbb{Q}$-divisor such that $(H+E_0) \cdot F_0 = 0$. Since $E_0 \cdot F_1=0$, there is a positive rational number $b$ such that $(H+E_0+bE_1) \cdot F_1 = 0$. Moreover $(H+E_0+bE_1) \cdot F_0 = 0$ since $E_1 \cdot F_0=0$. Let $A$ be a sufficiently ample divisor on $S$. Then $D=p_2^*A+H+E_0+bE_1$ is a nef and big $\mathbb{Q}$-divisor. Moreover, any curve $B$ have intersection number $0$ with $D$ must be contract by $T_2 \to S$ since $A$ is sufficiently ample. The curve $B$ is also contained in $E_0 \cup E_1$ since $H$ is ample. Since $K_{T_2} \cdot F_0 =K_{T_2} \cdot F_1 = 0$, there exists a large integer $k$ such that $kD-K_{T_2}$ is nef and big.  Then by the basepoint-free theorem (see \cite[Thm. 3.3]{KM98}), there is a large enough integer $a$ such that $aD$ is Cartier and $|aD|$ is basepoint-free.  The linear system $|aD|$ induces a contraction which contracts $E_0$ and $E_1$.

By contracting $E_0$ and $E_1$, we obtain a threefold $T$. And there is an induced fibration $p:T \to S$ such that $p^*C = 2(p^*C)_{red}$. Notice that $T_2 \to T$ is a resolution of singularities and $K_{T_2}=f^*K_T$. Hence $V$ has canonical singularities.
\end{const}

Now we will construct the example.

\begin{exmp}
\label{over-surf-pf}
Let $C=\{[x:y:z] \in \mathbb{P}^2 \ | \  x^6+y^6+z^6=0 \}$ be a smooth curve in $\mathbb{P}^2$. Then $\mathscr{O}_{\mathbb{P}^2}(2K_{\p^2}+C) \cong \mathscr{O}_{\mathbb{P}^2}$ is effective. Let $X_0=\p^1 \times \p^2$. By the method of Construction \ref{const-non-reduced-fib}, we can construct a fiber space $p:X \to \p^2$ such that $p^*C = 2(p^*C)_{red}$. Then $X$ has  canonical singularities. Since general fibers of $p$ are isomorphic to $\p^1$, $X$ is rationally connected by \cite[Thm. 1.1]{GHS03}. Moreover, since $2(K_{\p^2}+\frac{1}{2}D)$ is an effective divisor and $p$ is equidimensional, we have $H^0(X, (\Omega_X^1)^{[\otimes 4]}) \neq \{0\}$ (see \S 4.2).
\end{exmp}

Note that in the three examples above, the variety $X$ we constructed is a rationally connected varieties with canonical singularities. The divisor $K_X$ is not pseudo-effective and we always have a Mori fibration from $X$ to a variety $Z$. The non-zero pluri-forms of $X$ come from the base $Z$. However, in the following example, the variety we construct is a rationally connected threefold with klt singularities whose canonical divisor is ample. Its non-zero pluri-forms come from its canonical divisor.

\begin{exmp}
\label{klt-type-gal}
Let $X_1$ be the Fermat hypersurface in $\mathbb{P}^4$ defined $x_0^6+x_1^6+ \cdots + x_4^6=0$. Then $X_1$ is a smooth threefold such that $K_{X_1}$ is ample. Moreover, the Picard number of $X_1$ is $1$ by the Lefschetz theorem (see \cite[Exmp. 3.1.23]{Lar04}). Let $G$ be the group $\mathbb{Z}/6\mathbb{Z}$ with generator $g$. Define an action of $G$ on $\p^4$ by $g\cdot [x_0:\cdots:x_4]=[\omega x_0:\omega x_1:x_2:x_3:x_4]$, where $\omega$ is a primitive 6th root of unity. This action induces an action of $G$ on $X_1$ which is free in codimension $1$. Denote the quotient $X_1/G$ by $X$ and the natural morphism $X_1 \to X$ by $\pi$. Then $X$ has $\mathbb{Q}$-factorial klt singularities but does not have canonical singularities (see \cite[Thm 3.1]{Reid79}). Now we will prove that $X$ is rationally connected. First we will prove $X$ is uniruled. If $w$ is a general point in $X$, then there is a point $w_1\in X_1$ such that $\pi(w_1) = w$ and the first two coordinates of $w_1$ are all non-zero. There are two hyperplanes in $\p^4$ passing through $w_1$,  $H_1=\{a_2x_2+a_3x_3+a_4x_4=0\}$ and $H_2=\{b_2x_2+b_3x_3+b_4x_4=0\}$, such that the intersection $C_1=H_1\cap H_2 \cap X_1$ is a smooth curve of genus $10$. Moreover, there are exactly $6$ points on $C_1$, given by $X_1 \cap \{x_2=x_3=x_4=0\}$, which are fixed under the action of $G$. Hence $\pi|_{C_1}:C_1 \to \pi(C_1)=C$ is ramified exactly at those $6$ points with degree $6$. By adjunction formula, we have the genus of $C$,  $g(C)=1+2^{-1} \times 6^{-1} \times (2\times 10 -2 -6 \times 5)=0$. Hence $C$ is a smooth rational curve. This implies that $X$ is uniruled. From the lemma below, we conclude that $X$ is rationally connected.
\end{exmp}

\begin{lemma}
\label{pic-1-rc}
Let $X$ be a normal projective variety which is $\mathbb{Q}$-factorial. Assume that $\mathrm{Pic}_{\mathbb{Q}}(X) \cong \mathbb{Q}$. Then $X$ is rationally connected if and only if it is uniruled.
\end{lemma}

\begin{proof}
Note that if $X$ is smooth, this was proven in \cite[Cor. 2.4]{cam92}. Assume that $X$ is uniruled and let $\pi: X \dashrightarrow Z$ be the MRC fibration for $X$ (see \cite[Thm. 2.3]{cam92}). We will argue by contradiction. Assume that dim$X>$dim$Z>0$. By \cite[Thm. 2.3]{cam92}, there is a Zariski open subset $Z_0$ of $Z$ such that $\pi$ is regular and proper over $Z_0$. Let $C$ be a curve contained in some fiber of $\pi|_{Z_0}$. Let $D$ be a non-zero effective Cartier divisor on $Z_0$. Let $E_0=\pi^*D$. Then $E_0$ is a non-zero effective divisor on $X_0=\pi^{-1}Z_0$. Let $E$ be the closure of $E_0$ in $X$, then $E$ is a non-zero effective Weil divisor on $X$. However, the intersection number of $E$ and $C$ is zero, which contradicts $\mathrm{Pic}_{\mathbb{Q}}(X) =\mathbb{Q}$.
\end{proof}

\section{Fibrations over curves}
\label{dimZ=1}

\subsection{Construction of rationally connected varieties carrying non-zero pluri-forms}

In this subsection, we will give a method to construct a rationally connected varieties which carry non-zero pluri-forms (Theorem \ref{thm_curve}). Together with Theorem \ref{terminal-3-fold}, we will see that every rationally connected threefold with $\mathbb{Q}$-factorial terminal singularities which carries non-zero pluri-forms can be constructed by this method.

From Lemma \ref{fibration-section}, we know that if general fibers of an equidimensional fibration $p:X \to Z$  do not carry any non-zero pluri-form, then $H^0(X,(\Omega_X^1)^{[\otimes m]}) \cong H^0(X, ((p^*\Omega_Z^1)^{sat})^{[\otimes m]})$ for all $m>0$. Moreover, $H^0(X, ((p^*\Omega_Z^1)^{sat})^{[\otimes m]}) \cong  H^0(Z,p_*((p^*\Omega_Z^1)^{sat})^{[\otimes m]}))$. Hence we would like to know what $p_*((p^*\Omega_Z^1)^{sat})^{[\otimes m]})$ is. In the case where $Z$ is a smooth curve, this is not difficult to compute. Notice that if $R$ is the ramification divisor of $p$, then $((p^*\Omega_Z^1)^{sat})^{[\otimes m]} \cong (p^*\Omega_Z^1)^{\otimes m} \otimes \sO_X(mR)$ for $m>0$. By the projection formula, we know that  $p_*((p^*\Omega_Z^1)^{sat})^{[\otimes m]}) \cong (\Omega_Z^1)^{\otimes m} \otimes p_*\sO_X(mR)$. Hence it is enough to compute $p_*\sO_X(mR)$.

\begin{lemma}
\label{pushdown-form}

Let $p:X \to Z$ be a fibration from a normal variety to a smooth curve. Let $z$ be a point in $Z$. Let $D=p^*z-(p^*z)_{red}$. Then $p_*\mathscr{O}_{X}(mD) \cong \mathscr{O}_Z([\frac{m(m(p,z)-1)}{m(p,z)}] z)$ for any $m\geqslant 0$.
\end{lemma}

\begin{proof}
The problem is local and we may assume that $Z$ is affine. We know that $p_*\mathscr{O}_{W}(mD)$ is an invertible sheaf. If $\theta$ is a section of $p_*\mathscr{O}_{W}(mD)$, then it is a rational function on $Z$ whose pullback on $W$ is a rational function which can only have pole alone $p^*z$. Hence $\theta$ can only have pole at $z$. Let $d_{\theta}$ be the degree of the pole, then the pullback of $\theta$ in $V$ is a section of $\mathscr{O}_{W}(mD)$ if and only if $m(p,z)d_{\theta} \leqslant m(m(p,z)-1)$, \textit{i.e.} $d_{\theta} \leqslant [\frac{m(m(p,z)-1)}{m(p,z)}]$.
\end{proof}

This relation between ramification divisor and pluri-forms gives us an idea of how to construct rationally connected varieties which carry non-zero pluri-forms. Notice that if $p:X \to \p^1$ is a fibration such that  general fibers of $p$ are rationally connected, then $X$ is rationally connected by \cite[Thm. 1.1]{GHS03}. Moreover, if general fibers of $p$ do not carry any non-zero pluri-form, then from the discussion above $X$ carries non-zero pluri-forms if and only if $(\Omega_{\p^1}^1)^{\otimes m} \otimes p_*\sO_X(mR)$ has non-zero sections for some $m>0$, where $R$ is the ramified divisor of $p$. However, $\Omega_{\p^1}^1 \cong \sO_{\p^1}(-2)$ and $ p_*\sO_X(mR) \cong \sO_{p^1}(\sum_{z\in \p^1}[\frac{(m(p,z)-1)m}{m(p,z)}])$ since any two points in $\p^1$ are linearly equivalent. Hence $X$ carries non-zero pluri-forms if and only if $\sum_{z\in \p^1}\frac{m(p,z)-1}{m(p,z)} \geqslant 2$.

Now we will try to construct this kind of varieties by taking quotients. Let $T$ be a normal projective variety, let $B$ be a smooth projective curve and let $G$ be a finite commutative group. Assume that there are actions of $G$ on $T$ and $B$. Assume that there is a $G$-equivariant fibration $q:T \to B$. Let $p_T:T \to T/G$, $p_B: B \to B/G$ be natural projections and let $h:T/G \to B/G$ be the induced fibration. Let $S_i$, $i=1,...,r$ be all the $G$-orbits in $B$ whose cardinality is less than $|G|$, the cardinality of $G$. Let $z_i$ be the image of $S_i$ under the map $B \to B/G$. Then the $z_i$'s are the points in $B/G$ over which $B\to B/G$ is ramified.  Let $G_i$ be the stabilizer of a point $b_i$ in $S_i$. Then $G_i$ acts on the set-theoretic fiber $q^{-1}\{b_i\}=F_i$. If $A_i$ is a component in $F_i$ and the stabilizer of a general point in $A_i$ has cardinality $d_i$, then $p_T$ is ramified along  $G_i\cdot A_i$ of degree $d_i$, where $G_i\cdot A_i$ is the orbit of $A_i$ under $G_i$. In this case, $p_{W} (A_i)$ has coefficient $\frac{e_if_i}{d_i}$ in $h^*z_i$, where $e_i$ is the coefficient of $A_i$ in $q^*b_i$ and $f_i$ is the cardinality of $G_i$. Denote min$\{\frac{e_if_i}{d_i}\ |  \  A_i \ \mathrm{a \ component \ in} \  F_i\}$ by $s_i$. Then $s_i$ is equal to $m(h,z_i)$.  We have the theorem below.

\begin{thm}
\label{thm_curve}
Let $T$ be a projective normal variety, $B$ be a projective smooth curve with positive genus and $G$ be a finite commutative group. Assume they satisfy the following:

\begin{enumerate}
\item There is a fibration  $q$ from $T$ to $B$ such that for any general fiber $F_{q}$ of $q$, $F_{q}$ is rationally connected and does not carry any non-zero pluri-form. Moreover, $m(q,b)=1$ for all $b \in B$.
\item There exist actions of $G$ on $B$ and on $T$ such that $q$ is $G$-equivariant.
\item The quotient $B/G$ is isomorphic to $\mathbb{P}^1$.
\end{enumerate}

Then the quotient $X=T/G$ is a normal projective rationally connected variety and there is a fibration $p:X \to \mathbb{P}^1$. Moreover, if we define $s_i$, $i=1,...,r$ as above, then $X$ carries non-zero pluri-forms if and only if $\Sigma_{i=1}^{r}\frac{s_i-1}{s_i} \geqslant 2$. In fact, $H^0(X, (\Omega^1_X)^{[\otimes m]}) \cong H^0(\mathbb{P}^1, \mathscr{O}_{\mathbb{P}^1}(-2m+\sum_{i=1}^{r}[\frac{(s_i-1)m}{s_i}]))$ for $m>0$.
\end{thm}

\begin{proof}
The fibration $p:X \to \p^1$ is induced by $q$. Since general fibers of $q$ are rationally connected, general fibers of $p$ are also rationally connected. Hence $X$ is rationally connected by \cite[Thm. 1.1]{GHS03}. If $F_{p}$ is a general fiber of $p$, then $F_p$ does not carry any non-zero pluri-form neither. Hence, $H^0(X, (\Omega^1_X)^{[\otimes m]}) \cong H^0(\mathbb{P}^1, \mathscr{O}_{\mathbb{P}^1}(-2m)\otimes p_*\sO_X(mR))$ for $m>0$, where $R$ is the ramification divisor. However, since $m(q,b)=1$ for all $b \in B$, we have $\sum_{i=1}^{r}[\frac{(s_i-1)m}{s_i}] = \sum_{z\in \p^1}[\frac{(m(p,z)-1)m}{m(p,z)}]$ by the definition of $s_i$. Finally, by Lemma \ref{pushdown-form}, we obtain $p_*\sO_X(mR) \cong \sO_{\p^1}(\sum_{i=1}^{r}[\frac{(s_i-1)m}{s_i}])$ and  $H^0(X, (\Omega^1_X)^{[\otimes m]}) \cong H^0(\mathbb{P}^1, \mathscr{O}_{\mathbb{P}^1}(-2m+\sum_{i=1}^{r}[\frac{(s_i-1)m}{s_i}]))$.
\end{proof}

\begin{rem}
\label{rem-thm-curve}
Conversely, Let $p:X \to \p^1$ be a fibration  such that general fibers of $p$ are rationally connected and do not carry non-zero pluri-form. If $X$ carries non-zero pluri-forms, then $X$ can be constructed by the method described above. In fact, by the discussion above, we have $\sum_{z\in \p^1}\frac{(m(p,z)-1)}{m(p,z)} \geqslant 2$. In particular, there are at least three points in $\p^1$ such that the multiplicity of $p$ is larger than $1$ over these points. Let $z_1,...,z_r$ be all the points in $\p^1$ over which  the multiplicity of $p$ is larger than $1$. Let $m_i=m(p,z_i)$. Since $r\geqslant 3$, there is a smooth curve $B$ and a Galois cover $p_B:B \to \p^1$ with Galois group $G$ such that $p_B$ is ramified exactly over the $z_i's$ and the degree of ramification is $m_i$ at each point over $z_i$ (see \cite[Lem. 6.1]{KO82}). Let $T$ be the normalisation of the fiber product $X\times_{\p^1}B$. Then we obtain a natural fibration $q:T \to B$ such that $m(q,b)=1$ for all $b\in B$. Moreover, $G$ acts naturally on $T$ and $T/G\cong X$.
\end{rem}

\subsection{The case of threefolds with canonical singularities}

If $p:X \to Z$ is a Mori fibration from a threefold which has $\mathbb{Q}$-factorial canonical singularities to a smooth curve, then general fibers of $p$ are Fano surfaces which have canonical singularities. These surfaces do not carry any non-zero pluri-form by the following proposition.

\begin{prop}
\label{vanish-fano}
If $S$ is a Fano surface which has canonical singularities, then $H^0(S,(\Omega^1_S)^{[\otimes m]}) = \{0\}$ for all $m>0$.
\end{prop}

\begin{proof}
Assume the opposite. If $S'$ is the result of a MMP for $S$, then $S'$ is a Mori fiber space. By \cite[Thm. 2.1]{Ou13}, $S'$ has Picard number $2$ and we have a Mori fibration $S' \to \p^1$. Let $p_S$ be the composition of $S \to S' \to \p^1$. By \cite[Thm1.4]{Ou13}, there is a smooth curve $B$ of positive genus and a finite morphism $B\to \p^1$ such that the natural morphism $S_B \to S$ is \'etale in codimension $1$, where $S_B$ is the normalisation of $S\times_{\p^1} B$. Hence $S_B$ is a Fano surface. Moreover it has canonical singularities (see \cite[Prop. 5.20]{KM98}). So it is rationally connected by \cite[Cor. 1.3 and 1.5]{HM07}. Hence $B$ is also rationally connected, which is a contradiction since its genus is positive.
\end{proof}

Thanks to this proposition, we obtain the following result. 

\begin{prop}
\label{can-over-curve}
Let $X$ be a rationally connected threefold with $\mathbb{Q}$-factorial canonical singularities. Assume that $X$ carries non-zero pluri-forms. Let $X'$ be the result of a MMP for $X$ and assume that there is a Mori fibration $p':X' \to \p^1$. Let $Y$ be the normalisation of the graph of $X\dashrightarrow X'$. Then  $Y$ can be constructed by the method of Theorem \ref{thm_curve}. After a contraction from $Y$, we can obtain $X$.
\end{prop}

\begin{proof}
By  Proposition \ref{vanish-fano} and Lemma \ref{fibration-section}, we know that every element in $H^0(X',(\Omega_{X'}^1)^{[\otimes m]})$ comes from the base $\p^1$, so does every element in $H^0(X,(\Omega_X^1)^{[\otimes m]})$ by Lemma \ref{inj-mmp}. But we only have a rational map $X \dashrightarrow X' \to \p^1$. However, by Proposition \ref{iso-graph}, we know that $H^0(X,(\Omega_X^1)^{[\otimes m]}) \cong H^0(Y,(\Omega_Y^1)^{[\otimes m]})$. Moreover, we have a natural fibration $g:Y\to \p^1$ which is the composition of $Y \to X' \to \p^1$. If $F_g$ is a general fiber of $g$, then there is a birational morphism $F_g \to S$, where $S$ is a general fiber of $p':X' \to \p^1$. Since $S$ does not carry any non-zero pluri-form by Proposition \ref{vanish-fano},  neither does $F_g$ by Lemma \ref{inj-mmp}. From Remark \ref{rem-thm-curve}, we know that $Y$ can be constructed by the method of Theorem \ref{thm_curve}. After a contraction  from $Y$, we obtain the variety $X$.
\end{proof}

\section{Mori fibrations and non-zero pluri-form}
\label{dimZ=2}

In this section, we study relations between Mori fibrations and non-zero pluri-forms. We will consider Mori fibrations from a normal threefold to a normal surface. First, we recall the definition and some properties of slopes. Let $X$ be a normal projective
$\mathbb{Q}$-factorial variety. Let $\alpha$ be a class of $1$-cycles in $X$. For a coherent sheaf  $\mathscr{F}$ of positive rank, we define the slope $\mu_{\alpha}(\mathscr{F})$ of $\mathscr{F}$ with respect to $\alpha$ by
\begin{displaymath}
\mu_{\alpha}(\mathscr{F}) := \frac{\mathrm{det}(\mathscr{F}) \cdot \alpha}{\mathrm{rank}(\mathscr{F})},
\end{displaymath}  
where $\mathrm{det}(\mathscr{F})$ is the reflexive hull of $\bigwedge^{\mathrm{rank}\mathscr{F}}\mathscr{F}$. A class of movable curves is a class of $1$-cycles which has non negative intersection number with any pseudo-effective divisor. If $\alpha$ is a class of movable curves, then $\mu^{max}_{\alpha}(\mathscr{F}) = \mathrm{sup}\{\mu_{\alpha}(\mathscr{G}) \ |\ \mathscr{G} \subseteq \mathscr{F} \ \mathrm{a \ 
coherent \ subsheaf \ of \ positive \ rank}\}$ is well defined (see \cite[p.42]{MP97}). For any coherent sheaf $\mathscr{F}$, there is a saturated coherent subsheaf $\mathscr{G} \subseteq \mathscr{F}$ such
that  $\mu^{max}_{\alpha}(\mathscr{F}) = \mu_{\alpha}(\mathscr{G})$. If $\sE$ and $\sF$ are two coherent sheaves of positive rank, then $\mu_{\alpha}(\sF \otimes \sE) = \mu_{\alpha}(\sF)+\mu_{\alpha}(\sE)$.

\subsection{General properties}

Consider a Mori fibration $p : X \to Z$ from a normal rationally connected threefold to a normal variety $Z$ of positive dimension. Assume that $X$ has $\mathbb{Q}$-factorial klt singularities. Let $D_1,...,D_k$ be all prime Weil divisors in $Z$ such that $m_i=m(p,D_i)>1$. Let $\D=\sum_{i=1}^k \frac{m_i-1}{m_i}D_i$. Then det$((p^*\Omega_Z^1)^{sat}) \cong \sO_X(p^*(K_Z+\D))$ (see Remark \ref{rem-pb-can-div}). Assume that $X$ carries non-zero pluri-forms and general fibers of $p$ do not carry any non-zero pluri-form. If dim$Z=1$, then $K_Z+\D$ is an effective $\mathbb{Q}$-divisor of degree $-2+\sum_{i=1}^k \frac{m_i-1}{m_i}$ on $\p^1$ (see \S 3). The aim of this subsection is to prove something analogue in the case where dim$Z=2$. We prove that if $K_Z+\D$ is not pseudo-effective, then there will be a fibration form $Z$ to $\p^1$. In this case, we have an induced fibration $X\to \p^1$ and we are in the same situation as in \S3 (see Lemma \ref{non-pf-pic-2}). In order to do this, we will run a MMP for the pair $(Z,\D)$. To this end, we prove the following proposition which implies that the pair $(Z,\D)$ is klt.

\begin{prop}
\label{ram-is-klt}
Let $p :X\to Z$ be a Mori fibration from a $\mathbb{Q}$-factorial klt quasi-projective variety $X$ to a normal variety $Z$. Let $D_1,...,D_k$ be pairwise distinct prime Weil divisors in $Z$ such that  $p^*D_i = m_i(p^*D_i)_{red}$ with $m_i \geqslant 2$. Then the pair $(Z,\sum_{i=1}^k \frac{m_i-1}{m_i}D_i)$ is klt.
\end{prop}

\begin{proof}
Let $D=D_1+\cdots +D_k$. By \cite[Lem. 5-1-5]{KMM85}, $Z$ is $\mathbb{Q}$-factorial. Notice that the problem is local in $Z$, we may assume that $Z$ is affine. Let $k_1$ be the smallest positive integer such that $k_1D_1$ is a Cartier divisor. By taking the $k_1$-th root of the function defining the Cartier divisor $k_1D_1$, we can construct a finite covering $c_{1,1}: Z_{1,1}\to Z$ which is \'etale in codimension $1$. Moreover, $c_{1,1}^*D_1$ is a prime Cartier divisor (see \cite[Prop.-Def. 1.11]{Mor88}). By taking the $m_1$-th root of the function defining the Cartier divisor $c_{1,1}^*D_1$, we can find a finite covering $c_{1,2}:Z_1 \to Z_{1,1}$ which is ramified exactly over $D_1$ with ramification degree $m_1$ (see \cite[\S 4.1B]{Lar04}). Let $c_1=c_{1,2} \circ c_{1,1}:Z_1 \to Z$. Then $c_1^*D_1=m_1(c_1^*D_1)_{red}$. Repeating this construction for the strict transform of $D_2$ in $Z_1$, we have a finite morphism $Z_2\to Z_1$ which is ramified over the strict transform of $D_2$ in $Z_1$ with degree $m_2$. Let $c_2:Z_2 \to Z$ be the composition of $Z_2\to Z_1 \to Z$. Then $c_2$ is ramified over $D_1+D_2$ and $c_2^*D_i=m_i(c_1^*D_i)_{red}$ for $i=1,2$. By induction, we can construct a finite morphism $c_k:Z_k \to Z$  which is ramified over $D$ such that $c_k^*D_i=m_i(c_1^*D_i)_{red}$ for all $i$. We have $K_{Z_k} = c_k^*(K_Z+\sum_{i=1}^k \frac{m_i-1}{m_i}D_i)$ is $\mathbb{Q}$-Cartier (for the proof of this formula, see Lemma \ref{calcul-multi-div} below). Let $X_k$ be the normalisation of $X\times_Z Z_k$. Then the natural projection $c_X:X_k \to X$ is \'etale in codimension 1 and $K_{X_k}=c_X^*K_X$. Hence $X_k$ is klt by \cite[Prop. 5.20]{KM98}. Moreover, $X_k \to Z_k$ is a Fano fibration since $X\to Z$ is a Mori fibration. Hence there is a $\mathbb{Q}$-divisor $\D_k$ such that the pair $(Z_k, \D_k)$ is klt by \cite[Cor. 4.7]{Fuj99}. Since $K_{Z_k}$ is $\mathbb{Q}$-Cartier, $Z_k$ is klt (see \cite[Cor. 2.35]{KM98}). So the pair $(Z,\sum_{i=1}^k \frac{m_i-1}{m_i}D_i)$ is also klt by \cite[Prop. 5.20]{KM98}.
\end{proof}

In the remaining of this subsection, our aim is to prove the following theorem.

\begin{thm}
\label{dim2-picard-1}
Let $p : X \to Z$ be a Mori fibration from a projective normal threefold to a normal projective surface. Assume that $X$ has $\mathbb{Q}$-factorial klt singularities. Let $D_1,...,D_k$ be all prime Weil divisors in $Z$ such that $m_i=m(p,D_i)>1$. Let $\D=\sum_{i=1}^k \frac{m_i-1}{m_i}D_i$. If $X$ carries non-zero pluri-forms and $K_Z+\D$ is not pseudo-effective, then the result $Z'$ of any MMP for the pair $(Z,\D)$ has Picard number $2$.
\end{thm}

First we would like to illustrate the idea of the proof in a simple case. Let $f_Z:(Z,\D)\to (Z',\D')$ be the result of a MMP for the pair $(Z,\D)$. Assume in a first stage that $Z=Z'$. We will argue by contradiction. Assume that $Z'$ has Picard number $1$.  Since general fibers of $p$ are isomorphic to $\p^1$, they don't carry any non-zero pluri-form. By Lemma \ref{fibration-section}, the non-zero pluri-forms of $X$ come from $(p^*\Omega_Z^1)^{sat}$. Since $K_Z+\Delta$ is not pseudo-effective, we can prove that there is a rank $1$ reflexive subsheaf $\sH$ of $(p^*\Omega_Z^1)^{sat}$  such that $\sH^{[\otimes l]}$ is an invertible sheaf which has non-zero global sections for some positive integer $l$. Next, we can prove that $\sH^{[\otimes l]}$ is isomorphic to $\sO_X$ under the assumption that Pic$_{\mathbb{Q}}Z=\mathbb{Q}$ (by using Lemma \ref{subsheave-Kor-dim}). Let $W\to X$ be the cyclic cover with respect to $\sH$ and $l$. Let $W \to V \to Z$ be the Stein factorisation, then $V$ will be a klt log Fano variety (by Lemma \ref{base-normalisation}). Moreover $H^0(V,\Omega_V^{[1]})\neq \{0\}$ (we use Lemma \ref{picard-1-diff}). This contradicts Theorem \ref{klt-thmb}. 

However, in general $Z'$ is different from $Z$ and the fibration $X \to Z'$ is not equidimensional. The idea of the proof will be the same as above but details will be more complicated. We will work over an open subset $Z_0$ of $Z$ such that $f_Z|_{Z_0}$ is an isomorphism and codim $Z'\backslash f_Z(Z_0) \geqslant 2$. For the complete proof of the theorem, we will need several lemmas.

\begin{lemma}
\label{calcul-multi-div}
Let $p:X\to Z$ be an equidimensional morphism between smooth varieties. Let $n$ be the dimension of $X$ and $d$ be the dimension of $Z$. Let $D$ be a prime divisor in $Z$ and $E$ be a prime divisor in $X$ such that $E$ is a component of $p^*D$. Assume that the coefficient of $E$ in $p^*D$ is $k$. Then for any general point $x\in E$, there is an open neighbourhood $U\subseteq X$ of $x$ such that $(p^*\Omega_Z^d)^{sat}|_U \cong \sO_X(p^*K_Z+(k-1)E)|_U$, where  $(p^*\Omega_Z^d)^{sat}$ is the saturation of $p^*\Omega_Z^d$ in $\Omega_X^d$.
\end{lemma}

\begin{proof}
We may assume that $E$ is smooth around $x$ and $D$ is smooth around $p(x)$. There exist local coordinates $(a_1,a_2,...,a_n)$ and $(b_1,b_2,...,b_d)$ of $X$ and $Z$ around $x$ and $p(x)$ such that $E$ is defined by $\{a_1=0\}$, $D$ is defined by $\{b_1=0\}$ and $p$ is given by $(a_1,a_2,...,a_n) \mapsto (a_1^k,a_2,...,a_d)$. With these coordinates, the natural morphism $p^*\Omega_Z^1 \to \Omega_X^1$ is given by $(\mathrm{d}b_1,\mathrm{d}b_2,...,\mathrm{d}b_d) \mapsto (ka_1^{k-1}\mathrm{d}a_1,\mathrm{d}a_2,...,\mathrm{d}a_d)$ and the image of the natural morphism $p^*\Omega_Z^d \to \Omega_X^d$ is generated by $p^*(\mathrm{d}b_1\wedge \mathrm{d}b_2\wedge \cdots \wedge \mathrm{d}b_d)=(ka_1^{k-1}\mathrm{d}a_1)\wedge \mathrm{d}a_2 \wedge \cdots \wedge \mathrm{d}a_d$. Hence $(p^*\Omega_Z^d)^{sat}$ is generated by $\mathrm{d}a_1\wedge \mathrm{d}a_2 \wedge \cdots \wedge \mathrm{d}a_d$. Since $\{a_1=0\}$ defines the divisor $E$, we have $(p^*\Omega_Z^d)^{sat}|_U \cong \sO_X(p^*K_Z+(k-1)E)|_U$ for some open neighbourhood $U$ of $x$.
\end{proof}

\begin{lemma}
\label{pullback-can-div}
Let $p:X\to Z$ be an equidimensional morphism between normal varieties. Assume that $Z$ is $\mathbb{Q}$-factorial. Let $D_1,...,D_r$ be all the prime divisors in $Z$ such that  $m(p,D_i)$ is larger than $1$. Write $\mathrm{det}((p^*(\Omega_Z^1))^{sat}) \cong \sO_X(M)$ where $M$ is a divisor on $X$, then $M - (p^*(K_Z+\sum_{i=1}^{r}\frac{m(p,D_i)-1}{m(p,D_i)}D_i))$ is effective.
\end{lemma}

\begin{proof}
By Proposition \ref{refle-cod-2}, we only have to prove the assertion on an open subset of $X$ whose complement is of codimension at least $2$. Hence we may assume that both $X$ and $Z$ are smooth and that $\sum_{i=1}^rD_i$ is smooth. In this case, by Lemma \ref{calcul-multi-div}, the divisor  $M$ is linearly equivalent to $p^*K_Z + \sum_{i=1}^{r}\sum_{j=1}^{s_i}(n_{i,j}-1)E_{i,j}$ where the $E_{i,1},...,E_{i,s_i}$ are the components of $p^*D_i$ and $n_{i,j}$ is the coefficient of $E_{i,j}$ in $p^*D_i$. Since $m(p,D_i)$ is the smallest integer among $n_{i,1},...,n_{i,s_i}$, we have $n_{i,j}-1\geqslant \frac{m(p,D_i)-1}{m(p,D_i)} \cdot n_{i,j}$. Thus $\sum_{i=1}^{r}\sum_{j=1}^{s_i}(n_{i,j}-1)E_{i,j} \geqslant \sum_{i=1}^r\frac{m(p,D_i)-1}{m(p,D_i)}p^*D_i$.
\end{proof}

\begin{rem}
\label{rem-pb-can-div}
With the notation above, if in addition $p$ is a Mori fibration, then $p^*D_i$ is irreducible for all $i$ since $p$ has relative Picard number 1. In this case we obtain that $\mathrm{det}((p^*(\Omega_Z^1))^{sat})$ is isomorphic to $\sO_X(p^*(K_Z+\sum_{i=1}^{r}\frac{m(p,D_i)-1}{m(p,D_i)}D_i))$.
\end{rem}

\begin{lemma} 
\label{push-down-2}
Let $p:X\to Z$ be an equidimensional fibration between normal varieties. Let $(p^*\Omega_Z^1)^{sat}$ be the saturation of the image of $p^*\Omega_Z^1$ in $\Omega^{[1]}_X$. Then $p_*(((p^*\Omega_Z^1)^{sat})^{[\wedge r]}) \cong \Omega_Z^{[r]}$ for $r>0$.
\end{lemma}

\begin{proof}
Fix $r>0$. Let $\sM = p_*(((p^*\Omega_Z^1)^{sat})^{[\wedge r]})$. Let $X_0$ be the smooth locus of $X$ and let $i:X_0 \to X$ be the canonical injection. Then $\Omega_X^{[1]} = i_*\Omega_{X_0}^1$ by Proposition \ref{refle-cod-2}. Since codim$X\backslash X_0 \geqslant 2$ and $Z$ is smooth in codimension $1$, there is a smooth open subset $Z_0$ in $Z$ which is contained in $p(X_0)$ and codim$Z\backslash Z_0 \geqslant 2$. By Proposition \ref{refle-cod-2}, it is enough to prove that $\sM|_{Z_0} \cong \Omega_Z^{[r]}|_{Z_0}= \Omega_{Z_0}^r$. Hence, by replacing $Z$ by $Z_0$, we may assume that $Z$ is smooth and $p(X_0)=Z$. Let $D$ be the sum of the prime divisors in $Z$ whose pull-back by $p$ are not reduced divisors. There is an open subset $Z_1$ of $Z$ such that $D|_{Z_1}$ is snc and codim$Z\backslash Z_1 \geqslant 2$. From Proposition \ref{refle-cod-2}, by replacing $Z$ by $Z_2$, we may assume that $D$ is a snc divisor on $Z$.

First we will prove that there is a natural injection from $\Omega_Z^r$ to $\sM$. In fact, on $X_0$, we have an injection from $p^*\Omega_Z^1|_{X_0}$ to $(p^*\Omega_Z^1)^{sat}|_{X_0}$. Hence there is an injection from $(p^*\Omega_Z^1)^{\wedge r}|_{X_0}$ to $((p^*\Omega_Z^1)^{sat})^{[\wedge r]}|_{X_0}$ on $X_0$. Since  $(p^*\Omega_Z^1)^{\wedge r}$ is without torsion, we have an injection from $(p^*\Omega_Z^1)^{\wedge r}$ to $((p^*\Omega_Z^1)^{sat})^{[\wedge r]}$ on $X$. Since $p(X)=Z$, we obtain an injection from $p_*((p^*\Omega_Z^1)^{\wedge r})$ to $\sM$. By the projection formula, this implies that $\Omega_Z^r \subseteq \sM$.

Now we will prove that $\Omega_Z^r \cong \sM$. Let $W=p^{-1}(Z \backslash D)$. Then the morphism $p|_W: W \to Z\backslash D$ is smooth in codimension $1$. Thus, $((p^*\Omega_Z^1)^{sat})^{[\wedge r]}|_W \cong p^*\Omega_Z^r|_W$ (see the proof of Lemma \ref{calcul-multi-div}). Then we obtain $\sM|_{Z\backslash D} \cong \Omega_{Z\backslash D}^{r}$ by the projection formula. Let $U$ be any open set in $Z$ and let $\beta$ be any element of $\sM(U)$, \textit{i.e.} a section of $\sM|_U$. Then $\beta$ is a rational section of $\Omega_Z^r|_U$ which can only have pole along $D$ since $\sM|_{Z\backslash D} \cong \Omega_{Z\backslash D}^{r}$. However, by the definition of $\sM$, $\beta$ induces a regular section of $\Omega^{r}_{X_0}|_{p^{-1}(U)}$. This implies that $\beta$ does not have pole along $D$. Thus $\sM(U) = \Omega_Z^r(U)$. Hence  $\Omega_Z^r \cong \sM$.
\end{proof}

\begin{lemma}
\label{base-normalisation}
Let $p:X\to Z$ be an equidimensional fibration from a normal variety $X$ to a smooth variety $Z$. Let $D_1,...,D_r$ be all the prime divisors in $Z$ such that the multiplicity $m(p,D_i)$ is larger than $1$. If $c_X:X_1 \to X$ is a finite morphism which is \'etale in codimension $1$ and $X_1 \overset{p_1}{\longrightarrow} Z_1 \overset{c_Z}{\longrightarrow} Z $ is the Stein factorisation, then $K_{Z_1}\leqslant c_Z^*(K_Z+\sum_{i=1}^{r}\frac{m(p,D_i)-1}{m(p,D_i)}D_i)$. That is, there is an effective $\mathbb{Q}$-divisor $\D_1$ on $Z_1$ such that $K_{Z_1}+\D_1= c_Z^*(K_Z+\sum_{i=1}^{r}\frac{m(p,D_i)-1}{m(p,D_i)}D_i)$.
\end{lemma}

\begin{proof}
If $D$ is a prime divisor in $Z$, then $p_1^*c_Z^*D = c_X^*p^*D$. Let $E$ be any irreducible component of $c_Z^*D$. Since $c_X$ is \'etale in codimension $1$, by comparing the coefficients of $E$ in $p_1^*c_Z^*D = c_X^*p^*D$, the coefficient $m_E$ of $E$ in $c_Z^*D$ is not larger than $m(p,D)$. This implies that $(m_E-1) \leqslant ({m(p,D)-1})$. Now from Lemma \ref{calcul-multi-div}, we obtain $K_{Z_1}=c_Z^*K_Z+\Sigma_{j=1}^s(m_j-1)E_j$, where the $E_j's$ are all prime divisors along which $c_Z^*$ is ramified and $m_j$ is the degree of ramification of $c_Z$ along $E_j$. By the discussion above, we have $(m_j-1)  \leqslant (m(p,p(E_j))-1) = m(p,p(E_j)) \cdot \frac{m(p,p(E_j))-1}{m(p,p(E_j))}$. Hence $K_{Z_1}\leqslant c_Z^*(K_Z+\sum_{i=1}^{r}\frac{m(p,D_i)-1}{m(p,D_i)}D_i)$.
\end{proof}

\begin{lemma}
\label{subsheave-Kor-dim}
Let $p:X\to Z$ be a Mori fibration from a normal threefold $X$ to a smooth surface $Z$. Let $D_1,...,D_r$ be all the prime divisors in $Z$ such that the multiplicity $m_i=m(p,D_i)$ is larger than $1$. Assume that there is a projective $\mathbb{Q}$-factorial variety $V$ such that we have an open embedding $j:Z\to V$ with $\mathrm{codim}V\backslash Z \geqslant 2$. Assume further that the pair $(V,\sum_{i=1}^{r}\frac{m_i-1}{m_i}\bar{D}_i)$ is klt, where $\bar{D}_i$ is the closure of $\bar{D}_i$ in $V$. If $\sF$ is a rank $1$ reflexive subsheaf of $\Omega_X^{[1]}$ such that $\sF^{[\otimes m_0]} \cong \sO_X(p^*D_0)$ for some positive integer $m_0$ and some divisor $D_0$ on $Z$, then $\bar{D}_0$, the closure of $D_0$ in $V$, is not ample.
\end{lemma}

\begin{proof}
Assume the opposite. By replacing $m_0$ with a large multiple, we may assume that $\bar{D}_0$ is very ample. We may also assume that both $\bar{D}_0$ and $p^*D_0$ are prime and the pair $(V,\sum_{i=0}^{r}\frac{m_i-1}{m_i}\bar{D}_i)$ is klt. Let $c_X:X_1 \to X$ be the ramified cyclic cover with respect to $\sF$, $m_0$ and $p^*D_0$ (see \cite[Def. 2.52]{KM98}). Then $c_X$ is ramified over $p^*D_0$ with degree $m$ and $(c_X^*\sF)^{**} \cong \sO_{X_1}(E)$, where $E=(c_X^*p^*D_0)_{red}$. Moreover, over the smooth locus of $X$, there is an injection of sheaves from $c_X^*\Omega_X^1$ to $\Omega_{X_1}^1$. Hence by Proposition \ref{refle-cod-2}, we have an injection $\sO_{X_1}(E) \hookrightarrow \Omega_{X_1}^{[1]}$.

Let $X_1 \overset{p_1}{\longrightarrow} Z_1 \overset{c_Z}{\longrightarrow} Z $ be the Stein factorisation. Let $c_V:V_1 \to V$ be the normalisation of $V$ in the function field of $X_1$. Then we obtain an open embedding $j_1:Z_1 \to V_1$ such that $\mathrm{codim}V_1\backslash Z_1 \geqslant 2$. If $F_p$ is a general fiber of $p$ and if $F_{p_1}$ is a general fiber of $p_1$ which is mapped to $F_p$, then  $F_{p_1}\to F_p$ is \'etale. Since $p$ is a Mori fibration, $F_p$ is a smooth rational curve which is simply connected. Hence $F_{p_1}\to F_p$ is an isomorphism and $F_{p_1} \cong \p^1$. Hence $c_Z$ is of degree $m_0$. Let $H=(c_Z^*D_0)_{red}$. Since $c_Z\circ p_1:X_1 \to Z$ has connected fibers over $D_0$, we have $c_Z^*D_0=m_0H$. Hence $E_1=p_1^*H$. Note that $H$ is Cartier in codimension $1$. By shrinking $Z$, we may assume that $H$ is Cartier on $Z$.

Note that $c_X|_{X_1\backslash E}$ is \'etale in codimension $1$ and $c^*_ZD_0=m_0H$. From Lemma \ref{base-normalisation}, we conclude that there is an effective $\mathbb{Q}$-divisor $\D_1'$ in $Z_1$ such that $K_{Z_1}+\D_1' = p^*(K_Z+\sum_{i=0}^{r}\frac{m_i-1}{m_i}D_i)$. Hence if $\D_1$ is the closure of $\D_1'$ in $V_1$, then $K_{V_1}+\D_1 = p^*(K_V+\sum_{i=0}^{r}\frac{m_i-1}{m_i}\bar{D}_i)$. This implies that $(V_1,\D_1)$ is klt by \cite[Prop. 5.20 and Cor. 2.35]{KM98}.

Since general fiber of $p_1$ does not carry any non-zero pluri-forms,  by Lemma \ref{fibration-section}, the injection $\sO_{X_1}(E_1) \hookrightarrow \Omega_{X_1}^{[1]}$ factorises through $(p_1^*\Omega_{Z_1}^1)^{sat}$. Hence by the projection formula, we have an injection from $\sO_{Z_1}(H)$ to $p_{1*}((p_1^*\Omega_{Z_1}^1)^{sat})$. However, by Lemma \ref{push-down-2}, $p_{1*}((p_1^*\Omega_{Z_1}^1)^{sat})=\Omega_{Z_1}^{1}$. By Proposition \ref{refle-cod-2}, we get an injection from $\sO_{V_1}(\bar{H})$ to $\Omega_{V_1}^{[1]}$, where $\bar{H}$ is the closure of $H$ in $V_1$. Hence by Bogomolov-Sommes theorem (see \cite[Thm. 7.2]{GKKP11}), the Kodaira dimension of $\sO_{V_1}(\bar{H})$ is not larger than $1$. However, since $\bar{D}_0$ is ample, $\bar{H}$ is ample and has Kodaira dimension equal to dim$V_1=2$. This leads to a contradiction.
\end{proof}

\begin{lemma}
\label{picard-1-diff}
Let $p:X \to Z$ be an equidimensional fibration such that  general fibers of $p$ do not carry any non-zero pluri-form. Assume that $Z$ is smooth and there is an open embedding $j:Z \to V$ such that $\mathrm{codim}V\backslash Z \geqslant 2$.  
Assume that $H^0(V,\Omega_V^{[r]})=\{0\}$ for all $r>0$. Then $H^0(X,\Omega_X^{[r]}) = \{0\}$ for all $r>0$.
\end{lemma}

\begin{proof}
Assume the opposite. Let $\sF = (p^*\Omega_Z^1)^{sat}$ be the saturation of the image of $p^*\Omega_Z^1$ in $\Omega_X^{[1]}$. Since general fibers of $p$ do not carry any non-zero pluri-form, we have $H^0(X,\Omega_X^{[r]}) \cong H^0(X,\sF^{[\wedge r]})$ by Lemma \ref{fibration-section}. Hence there is an injection from $\sO_X$ to $\sF^{[\wedge r]}$. By taking the direct image, we have an injection from $\sO_Z$ to $p_*(\sF^{[\wedge r]})$. By Lemma \ref{push-down-2}, $p_*(\sF^{[\wedge r]}) \cong \Omega_Z^r$. This implies that $H^0(V,\Omega_V^{[r]}) \neq \{0\}$ by Proposition \ref{refle-cod-2}, which is a contradiction.
\end{proof}

Now we are ready to prove Theorem \ref{dim2-picard-1}.

\begin{proof}[Proof of \ref{dim2-picard-1}]
We will argue by contradiction. Let $f_Z: Z \to Z'$ be the result of a $(K_Z+\D)$-MMP. Assume that $Z'$ has Picard number $1$. If $\D' = f_{Z*}\D$, then $K_{Z'}+\D'$ is not pseudo-effective neither. Thus $-(K_{Z'}+\D')$ is ample. We know that there is an open subset $Z'_0 \subseteq Z'$ with $\mathrm{codim}Z'\backslash Z'_0 \geqslant 2$ such that $f_Z^{-1}$ is an isomorphism from $Z'_0$ onto its image. Let $Z_0$ be $f_Z^{-1}(Z'_0)$, which is an open subset in $Z$. Since $Z'$ is normal, by shrinking $Z'_0$ if necessary, we may assume that $Z_0\cong Z'_0$ are  smooth varieties. Let $X_0=p^{-1}(Z_0)$.

Since  $\mathrm{codim}Z'\backslash Z'_0 = 2$, there is a projective curve $C'_0$ in $Z'_0$ such that it  is an ample divisor in $Z'$. Let $C_0$ be the image of $C_0'$ in $Z_0$. Let $\alpha$ be a class of movable curves in $X$ whose image in $Z$ is not zero and is proportional to the class of $C_0$.

Since general fibers of $p$ do not carry any non-zero pluri-forms, by Lemma \ref{fibration-section}, we have $H^0(X,(\Omega_X^1)^{[\otimes m]}) \cong H^0(X,((p^*\Omega_Z^{1})^{sat})^{[\otimes m]})$ for any $m>0$, where $(p^*\Omega_Z^{1})^{sat}$ is the saturation of the image of  $(p^*\Omega_Z^{1})$ in $\Omega_X^{[1]}$. Hence $\mu^{max}_{\alpha} ((p^*\Omega_Z^{1})^{sat}) \geqslant 0$ and there is a coherent sheaf $\mathscr{H}$ saturated in $(p^*\Omega_Z^{1})^{sat}$ such that $\mu_{\alpha}(\mathscr{H})  \geqslant 0$.

However, since $p|_{X_0}$ is a Mori fibration and $Z_0$ is smooth, the $\mathbb{Q}$-Cartier divisor associates to the determinant of $(p^*\Omega_Z^{1})^{sat}|_{X_0}$ is equal to $p^*(K_Z+\D)|_{X_0}$ by Remark \ref{rem-pb-can-div}. Hence $\mathrm{det}((p^*\Omega_Z^{1})^{sat}) \cdot {\alpha} =p^*(K_Z+\D) \cdot \alpha= (p\circ f_Z)^*(K_{Z'}+\D') \cdot \alpha<0$. This implies that $\mathscr{H} \neq (p^*\Omega_Z^{1})^{sat}$. If $\mathscr{J}$ is the quotient $(p^*\Omega_Z^{1})^{sat}/\mathscr{H}$, then rank$\mathscr{J}=$rank$\sH=1$ and $\sH [\otimes] \sJ =$ det$((p^*\Omega_Z^{1})^{sat})$. Thus $\mathscr{J} \cdot \alpha < 0$.

Since $H^0(X,((p^*\Omega_Z^{1})^{sat})^{[\otimes m]}) \neq \{0\}$ for some $m>0$, we have $H^0(X,\mathscr{H}^{\otimes s} [\otimes] \mathscr{J}^{\otimes t}) \neq \{0\}$ for some $s,t \geqslant 0$ by Lemma \ref{long-filtration}. In this case, we have $s>t$ for $(\sH [\otimes] \sJ) \cdot \alpha<0$ and  $\mathscr{J} \cdot \alpha < 0$.

Let $F_{p}$ be a general fiber of $p$. Then the class of  $F_{p}$ is movable and $\mu_{F_{p}}(\mathscr{H}^{[\otimes s]} [\otimes] \mathscr{J}^{[\otimes t]}) \geqslant 0$. Moreover, we have $(\mathscr{H} [\otimes]\mathscr{J}) \cdot F_p = \mathrm{det}((p^*\Omega_Z^{1})^{sat}) \cdot F_p=0$. This implies that ${\mathscr{H}} \cdot F_{p} \geqslant 0$ since $s>t$. However, since the restriction of $(p^*\Omega_Z^{1})^{sat}$ on $F_{p}$ is isomorphic to $\mathscr{O}_{F_{p}} \oplus \mathscr{O}_{F_{p}}$, we have ${\mathscr{H}} \cdot F_{p}=0$. 

Let $k$ be the smallest positive integer such that $\sH^{[\otimes k]}$ is invertible. Then there is a Cartier divisor $L$ in $Z$ such that $\sH^{[\otimes k]} \cong \sO_{X}(p^*L)$. Let $L'=f_{Z*}$L. Then $L'\cdot C'_0 \geqslant 0$ since $\mu_{\alpha}(\sH) \geqslant 0$. Notice that if $L'\cdot C'_0 > 0$, then $L'$ is ample on $Z'$. Hence by Lemma \ref{subsheave-Kor-dim}, we can only have $L' \cdot C_0'=0$ and $L'$ is numerically equal to the zero divisor since $Z'$ has Picard number $1$. By \cite[Lem. 2.6]{AD12}, there is a positive integer $k'$ such that $k'L'$ is linearly equivalent to the zero divisor. Hence $\sO_Z(k'L)|_{Z_0} \cong \sO_{Z_0}$ and $\sH^{[\otimes kk']}|_{X_0} \cong \sO_{X_0}$. Let $l$ be the smallest positive integer such that $\sH^{[\otimes l]}|_{X_0} \cong \sO_{X_0}$. Let $c:W_0 \to X_0$ be the cyclic cover with respect to $\sH|_{X_0}$ and $l$ (see \cite[Def. 2.52]{KM98}). Then $c$ is \'etale in codimension $1$. We have $\sO_{W_0} \cong (c^*\sH)^{**}$ and there is an injection $(c^*\sH)^{**} \hookrightarrow (c^*\Omega_{X_0}^1)^{**}\cong \Omega_{W_0}^{[1]}$. Hence $H^0(W_0,\Omega_{W_0}^{[1]})\neq \{0\}$.

On the other hand, let $W_0\to V_0 \to Z_0'$ be  the Stein factorisation. Let $h:V \to Z'$ be the normalisation of $Z'$ in the function field of $W_0$. Then there is an open embedding from $V_0$ to $V$ such that codim$V\backslash V_0 \geqslant 2$. By Lemma \ref{base-normalisation}, we have $K_{V_0} \leqslant (h^*(K_{Z'}+\D'))|_{V_0}$. Since codim$V\backslash V_0 \geqslant 2$, there is an effective $\mathbb{Q}$-divisor $\D_V$ in $V$ such that $K_{V}+\D_V= h^*(K_{Z'}+\D')$. Thus the pair $(V,\D_V)$ is klt by \cite[Prop. 5.20]{KM98}. Moreover, $-(K_{V}+\D_V)$ is ample since $-(K_{Z'}+\D')$ is ample, hence $V$ is rationally connected by \cite[Cor. 1.13]{HM07}. By Theorem \ref{klt-thmb}, we have $H^0(V,\Omega_V^{[1]})=0$. However, since general fibers of $W_0 \to V_0$ are isomorphic to $\p^1$ which does not carry any non-zero pluri-form, we conclude that $H^0(W_0,\Omega_{W_0}^{[1]})=\{0\}$ from Lemma \ref{picard-1-diff}. This is a contradiction.
\end{proof}

\subsection{The case of threefolds with canonical singularities}

In this subsection, we will study Mori fibrations $X \to Z$ such that $X$ is a projective threefold with $\mathbb{Q}$-factorial canonical singularities and $Z$ is a projective normal surface. Assume that $X$ carries non-zero pluri-forms. Then by Theorem \ref{dim2-picard-1}, either  the result $Z'$ of any $(K_Z+\D)$-MMP has Picard number $2$ or $(K_Z+\D)$ is pseudo-effective, where $\D$ the $\mathbb{Q}$-divisor defined in  Theorem \ref{dim2-picard-1}. We will study the first case in Proposition \ref{non-pf-pic-2} and the second case in Proposition \ref{base-ps-eff}.

\begin{prop}
\label{non-pf-pic-2}
Let $p:X \to Z$ be a Mori fibration from a projective threefold to a projective surface such that $X$ has at most $\mathbb{Q}$-factorial canonical singularities. Let $\D$ be the divisor in $Z$ defined in Theorem \ref{dim2-picard-1}. Assume that  $K_Z+\D$ is not pseudo-effective. Let $f_Z:Z\to Z'$ be the result of a $(K_Z+\D)$-MMP and let $\D'$ be the strict transform of $\D$ in $Z'$. Then there is a $(K_{Z'}+\D')$-Mori fibration $\pi' : Z' \to \mathbb{P}^1$. Let $\pi = \pi' \circ f_Z$ and $q= \pi \circ p : X \to \mathbb{P}^1$. Then we have $H^0(X, (\Omega^1_{X})^{[\otimes m]}) \cong H^0(X, ((q^*\Omega^1_{\mathbb{P}^1})^{sat})^{[\otimes m]})$ for any $m>0$, where $(q^*\Omega^1_{\mathbb{P}^1})^{sat}$ is the saturation of the image of $q^*\Omega^1_{\mathbb{P}^1}$ in  $\Omega^{[1]}_{X}$.
\end{prop}

\begin{proof}
Let $\mathscr{H} = (q^*\Omega^1_{\mathbb{P}^1})^{sat}$ and let $\mathscr{F} = (p^*\Omega^{[1]}_{Z})^{sat}$ be the saturation of $p^*\Omega^{[1]}_{Z}$ in $\Omega^{[1]}_{X}$. Then we have an exact sequence of coherent sheaves $0 \to \mathscr{H} \to \mathscr{F} \to \mathscr{J} \to {0}$, where $\sJ$ is a torsion free sheaf such that det$\sF =\sH [\otimes] \sJ$. Let $\alpha$ be a class movable curves in $X$ whose image in $Z$ is not zero and is proportional to the class of general fibers of $\pi : Z \to \mathbb{P}^1$. Then $ \mathscr{H} \cdot \alpha = 0$. Moreover, we have $(\sH [\otimes] \sJ) \cdot \alpha= \mathrm{det}\mathscr{F} \cdot \alpha=  p^*(K_Z+\Delta) \cdot \alpha$ by Remark \ref{rem-pb-can-div}. This intersection number is negative since for general fibers $F_{\pi}$ of $\pi: \Z \to \mathbb{P}^1$, we have $(K_Z+\Delta) \cdot F_{\pi}=  (K_{Z'}+\Delta') \cdot (f_{Z*}F_{\pi})<0$. Hence $\sJ \cdot \alpha<0$.

There is an open subset $U$ of $X$ with $\mathrm{codim}(X \backslash U) \geqslant 2$ such that we have an exact sequence of locally free sheaves over $U$, $0\to \mathscr{H}|_U \to \mathscr{F}|_U \to \mathscr{J}|_U \to {0}$. Since $\mu_{\alpha}(\sH^{\otimes s} \otimes \sJ^{\otimes t})  <0$ if $t>0$, we have $H^0(U, \sH|_U^{\otimes s} \otimes \sJ|_U^{\otimes t}) =\{0\}$ if $t>0$. Hence by Lemma \ref{fibration-section} and Lemma \ref{exact-section}, we obtain $H^0(X, (\Omega^1_{X})^{[\otimes m]}) \cong H^0(X, (\mathscr{F})^{[\otimes m]}) \cong H^0(X, (\mathscr{H})^{[\otimes m]})$ for any $m>0$.
\end{proof}

\begin{exmp}
\label{exmp-surf-pic-2}
We will give an example of this kind of threefolds. Let $Z=\p^1 \times \p^1$. Denote by $pr_1$, $pr_2$  the two natural projections from $Z$ to $\p^1$. Let $z_1$,...,$z_r$ be $r\geqslant 4$ different points in $\p^1$ and let $C_i=pr_1^*z_i$  for $i=1,...,r$.  Let $X_0 = \p^1 \times Z$. By the method of Construction \ref{const-non-reduced-fib}, we can construct a Mori fibration $\pi:X \to Z$ such that $m(\pi,C_i)=2$ for $i=1,...,r$. Notice that $K_Z+\frac{1}{2}(C_1+\cdots +C_r)$ is not pseudo-effective since it has negative intersection number with general fibers of $pr_1$. Moreover, we have $H^0(X,(\Omega_X^1)^{[\otimes 2]}) \cong H^0(\p^1, (\Omega_{\p^1}^1)^{\otimes 2} \otimes \sO_{\p^1}(C_1+\cdots +C_r)) \cong H^0(\p^1, \sO_{\p^1}(-4+r)) \neq \{0\}$.
\end{exmp}

Now we will treat the second case. Note that this is the case for Example \ref{KB-pf} and Example \ref{over-surf-pf}.

\begin{prop}
\label{base-ps-eff}
Let $p:X\to Z$ be a Mori fibration from a projective threefold to a projective surface. Assume that $X$ has at most $\mathbb{Q}$-factorial klt singularities. Assume that $(K_Z+\D)$ is pseudo-effective, where $\D$ is the $\mathbb{Q}$-divisor defined in Theorem \ref{dim2-picard-1}. Then $X$ carries non-zero pluri-forms.
\end{prop}

\begin{proof}
By the abundance theorem for log surface (see. \cite{AFKM92}),  $(K_Z+\D)$ is effective . Hence there is  a positive integer $l$ such that $l(K_Z+\D)$ is an effective Cartier divisor. This implies that $H^0(X,(\Omega_X^1)^{[\otimes 2l]}) \neq \{0\}$ by the Lemma \ref{pullback-can-div}.
\end{proof}

\section{Proof of Theorem 1.4}
\label{Proof of Theorem 1.4}

In this section, we will complete the proof of Theorem \ref{terminal-3-fold}. First we will show that if $X$ is a rationally connected projective threefold with \textit{terminal} singularities such that $H^0(X,(\Omega_X^1)^{[\otimes m]}) \neq \{0\}$ for some $m>0$, then there is a dominant rational map from $X$ to $\p^1$ (Lemma \ref{ter-fibraiton-curve}). To this end, we will need the following lemma.

\begin{lemma}
\label{fiber-can-fibration}
Let $p:X \to Z$ be a Mori fibration such that $X$ has $\mathbb{Q}$-factorial terminal singularities and $\mathrm{dim}Z = \mathrm{dim}X-1$. Then there exists an open subset $Z_0 \subseteq Z$ with $\mathrm{codim}(Z \backslash Z_0) \geqslant 2$ such that every scheme-theoretic fiber over $Z_0$ is reduced.
\end{lemma}

\begin{proof}
Since $X$ has terminal singularities, it is smooth in codimension $2$ (see \cite[Cor. 5.18]{KM98}). There is a smooth open subset $Z_0 \subseteq Z$ with $\mathrm{codim}(Z \backslash Z_0) \geqslant 2$ such that $p|_{X_0}:X_0 \to Z_0$ is equidimensional and $X_0$ is smooth, where $X_0=q^{-1}(Z_0)$. Since $Z_0$ is smooth and $X$ is Cohen-Macaulay (see \cite[Thm. 5.10 and Thm. 5.22]{KM98}), fibers of $q$ over $Z_0$ are Cohen-Macaulay. Hence a fiber over $Z_0$ is reduced if and only if it is generically reduced. If $z$ is a point in $Z_0$, then there is a smooth curve $B$ in $Z_0$ passing through $z$ such that $Y=p^{-1}(B)$ is a smooth surface. Moreover, $p|_Y:Y \to B$ is a Fano fibration for $p$ is a Mori fibration. Thus $p|_Y$ has reduced fibers over $B$ by \cite[Prop. 3.5]{Ou13}. Hence the scheme-theoretic fiber $p^{-1}(\{z\})$ is reduced.
\end{proof}

\begin{lemma}
\label{ter-fibraiton-curve}
Let $X$ be a rationally connected projective threefold with \textit{terminal} singularities such that $H^0(X,(\Omega_X^1)^{[\otimes m]}) \neq \{0\}$ for some $m>0$. Let $f_X:X \dashrightarrow X'$ be the result of a MMP for $X$. Then there is a fibration $p':X' \to \p^1$.
\end{lemma}

\begin{proof}
Notice that $X'$ is a Mori fiber space since $K_X$ is not pseudo-effective by \cite[Cor. 1.11]{Kol96}. Then we have a Mori fibration $q':X' \to Z$, where $Z$ is a normal rationally connected variety. If dim$Z=1$, then we are done.

By Lemma \ref{inj-mmp}, we know that $X'$ carries non-zero pluri-forms. Hence dim$Z>0$ by \cite[Thm. 2.1]{Ou13}.  Assume that dim$Z=2$. Then $Z$ has at most canonical singularities by \cite[Cor. 1.2.8]{MoP08}. Hence $K_Z$ is not pseudo-effective by \cite[Cor. 1.11]{Kol96}. Moreover, by Lemma \ref{fiber-can-fibration}, $m(q',D)=1$ for any effective divisor $D$ on $Z$. Hence, by Theorem \ref{dim2-picard-1}, if $Z \to Z'$ is the result of a MMP for $Z$, then $Z'$ has Picard number $2$. Hence we have a Mori fibration $Z' \to \p^1$. Let $p'$ be the composition of $X' \to Z \to Z' \to \p^1$. Then $p'$ is a fibration from $X'$ to $\p^1$
\end{proof}

With the notation above, notice that any general fiber $F'$ of $p':X' \to \p^1$ is smooth rationally connected surface. Hence $F'$ does not carry any non-zero pluri-form. Let $U$ be the largest open subset in $X$ over which $f_X:X \dashrightarrow X'$ is regular. Then codim$_{X}X\backslash U\geqslant 2$, codim$_{X'}X'\backslash f_X(U)\geqslant 2$ and the rational map $p : X\dashrightarrow \p^1$ is regular over $U$. If $F$ is a general fiber of $U\to \p^1$, then $f_X(F) \subseteq F'$, where $F'$ is a general fiber of $p'$. Moreover codim$_{F'}F'\backslash f_X(F) \geqslant 2$.  Hence $f_X(F)$ does not carry any non-zero pluri-form and neither does $F$ by Lemma \ref{inj-mmp}. The following lemma shows that the rational map $X \dashrightarrow \p^1$ is regular. Moreover, general fibers of $X\to \p^1$ are birational to the ones of $X'\to p^1$ which are smooth Fano surfaces. Hence general fibers of $X\to \p^1$ are rationally connected.

\begin{lemma}
\label{ter-reg}
Let $X$ be a projective threefold with $\mathbb{Q}$-factorial terminal singularities. Assume that there is a non constant rational map $p:X \dashrightarrow \p^1$ which is regular over $U$ such that $\mathrm{codim}X \backslash U \geqslant 2$. Assume that general fibers of $U\to \p^1$ do not carry any pluri-form. If $H^0(X, (\Omega_{X}^1)^{[\otimes m]}) \neq \{0\}$  for some $m>0$, then $p$ is regular.
\end{lemma}

\begin{proof}
Let $\Gamma$ be the normalisation of the graph of $p$. Let $pr_1:\Gamma\to X$, $pr_2:\Gamma\to \p^1$ be the natural projections. Then there is a natural injection from $H^0(\Gamma, (\Omega_{\Gamma}^1)^{[\otimes m]})$ to  $H^0(X, (\Omega_X^1)^{[\otimes m]})$ by Lemma \ref{inj-mmp}. Let $\sigma \in H^0(X, (\Omega_X^1)^{[\otimes m]})$ be a non-zero element. Then $\sigma$ induces a rational section $\sigma_{\Gamma}$ of $(\Omega_{\Gamma}^1)^{[\otimes m]}$ on $\Gamma$. Let $E$ be a $pr_1$-exceptional divisor. Then there is a curve in $E$ which is contracted by $pr_2$ since dim$E=2>1$. Hence this curve is not contracted by $pr_1$ since the graph of $p$ is included in $X\times \p^1$ and the normalisation map is finite. Thus $pr_1(E)$ is a curve in $X$ and $X$ is smooth around the generic point of $pr_1(E)$ since $X$ is smooth in codimension $2$ (see \cite[Cor. 5.18]{KM98}). Hence $\sigma_{\Gamma}$ does not have pole along $E$. This implies that we have an isomorphism from $H^0(\Gamma, (\Omega_{\Gamma}^1)^{[\otimes m]})$ to $H^0(X, (\Omega_X^1)^{[\otimes m]})$  induced by $pr_1$.

Note that $pr_1^{-1}|_U$ induces an isomorphism from $U$ onto its image. If $F_U$ is a general fiber of $p|_U$, then $pr_1^{-1}(F_U)$ is an open subset of $F_{\Gamma}$, where $F_{\Gamma}$ is a general fiber of $pr_2:\Gamma \to \p^1$. Since $F_U$ does not carry any non-zero pluri-form, neither does $F_{\Gamma}$. Hence by Lemma \ref{fibration-section}, we have $H^0(\Gamma, (\Omega_{\Gamma}^1)^{[\otimes m]}) \cong H^0(\Gamma, ((pr_2^*{\Omega_{\p^1}^1})^{sat})^{[\otimes m]})$.

We will first prove that $p$ is regular in codimension $2$. Assume the opposite. Then there is a divisor $D$ in $\Gamma$ which is exceptional for $pr_1:\Gamma\to X$ and codim$_Xpr_1(D)=2$. Since $X$ is smooth in codimension $2$ (see \cite[Cor. 5.18]{KM98}), $X$ is smooth around the generic point of $pr_1(D)$. Thus there is a smooth quasi-projective curve $C$ in $D$ such that $\Gamma$ is smooth along $C$ and $pr_1(C)$ is a smooth point in $X$. Notice that $C$ is horizontal over $\p^1$ under the projection $pr_2:\Gamma\to \p^1$ for the same reason as above. Let $\sigma$ be a non-zero element in $H^0(\Gamma, (\Omega_{\Gamma}^1)^{[\otimes m]})$. By the exact sequence of locally free sheaves $\Omega_{\Gamma}^1|_C \to \Omega_C^1 \to 0$, $\sigma$ induce an element $\sigma_C$ in $H^0(C,(\Omega_C^1)^{\otimes m})$. On one hand, $C$ is horizontal over $\p^1$ and $\sigma$ is non-zero in  $H^0(\Gamma, ((pr_2^*{\Omega_{\p^1}^1})^{sat})^{[\otimes m]})$, we have $\sigma_C \neq 0$. On the other hand, $pr_1(C)$ is a smooth point in $X$. Hence we obtain $\sigma_C=0$ since $\sigma$ is the pullback of  certain element in $H^0(X, (\Omega_{X}^1)^{[\otimes m]})$. This is a contradiction.

Now we will prove that $p$ is regular. Let $F_1$ and $F_2$ be two different fibers of $U \to \p^1$. Then their closures in $X$ are two Weil divisors and their intersection is included in a closed subset of codimension at most $3$. Hence their intersection is empty since $X$ is $\mathbb{Q}$-factorial. This implies that $p$ is regular.
\end{proof}

Together with Lemma \ref{fibration-section} and Lemma \ref{pushdown-form}, we can conclude Theorem \ref{terminal-3-fold}.

\bibliographystyle{amsalpha}
\bibliography{threefold}

\end{document}